\documentclass[a4paper]{amsart}

\usepackage[backrefs,lite]{amsrefs}
\usepackage{url}

\makeatletter
\DeclareFontEncoding{LS1}{}{}
\DeclareFontSubstitution{LS1}{stix}{m}{n}
\DeclareMathAlphabet{\skr}{LS1}{stixscr}{m}{n}
\makeatother

\usepackage{amssymb}              
\usepackage{wasysym}              
\usepackage{lmodern,bm}              
\usepackage{stmaryrd}
\usepackage[all]{xy}
\usepackage{tikz}
\usepackage{graphicx}
\usetikzlibrary{calc} 

\usepackage{amsmath,amssymb,amsthm}
\usepackage{tikz, pgfplots, pgfplotstable}
\usetikzlibrary{calc}
\usetikzlibrary{decorations.markings}
\usetikzlibrary{positioning,arrows}

\CompileMatrices

\newtheorem{theorem}{Theorem}[section]

\newtheorem{lemma}[theorem]{Lemma}
\newtheorem{proposition}[theorem]{Proposition}
\newtheorem{corollary}[theorem]{Corollary}

\theoremstyle{definition}

\newtheorem{definition}[theorem]{Definition}
\newtheorem{construction}[theorem]{Construction}

\newtheorem{example}[theorem]{Example}
\newtheorem{remark}[theorem]{Remark}

\numberwithin{equation}{theorem}

\def\vector2#1#2{\left(\begin{array}{c} #1 \\ #2 \end{array}\right)}
\def\Cl{{\rm Cl}}

\def\CC{{\mathbb C}}
\def\KK{{\mathbb K}}
\def\TT{{\mathbb T}}
\def\ZZ{{\mathbb Z}}
\def\RR{{\mathbb R}}

\def\QQ{{\mathbb Q}}
\def\PP{{\mathbb P}}
\def\XX{{\mathbb X}}

\def\Mat{{\rm Mat}}

\def\Chi{{\mathbb X}}

\def\quot{/\!\!/}

\def\conv{{\rm conv}}

\def\im{{\rm im}}

\def\bangle#1{{\langle #1 \rangle}}

\def\eins{{\mathbf 1}}

\def\GL{{\rm GL}}

\def\GL{{\rm GL}}

\def\Spec{{\rm Spec}}

\def\cone{{\rm cone}}

\def\pr{{\rm pr}}

\def\lcm{{\rm lcm}}
\def\im{{\rm im}}

\def\lin{{\rm lin}}

\DeclareMathOperator{\vol}{\mathrm{vol}}
\DeclareMathOperator{\Eff}{\mathrm{Eff}}
\DeclareMathOperator{\Mov}{\mathrm{Mov}}
\DeclareMathOperator{\SAmple}{\mathrm{SAmple}}
\DeclareMathOperator{\Ample}{\mathrm{Ample}}

\title[K\"ahler-Einstein metrics,  K\"ahler-Ricci solitons, Sasaki-Einstein metrics \hspace*{.3cm}]{Log del Pezzo $\CC^*$-surfaces, K\"ahler-Einstein metrics, \\  K\"ahler-Ricci solitons and Sasaki-Einstein metrics}

\author[Daniel H{\"a}ttig, J\"urgen Hausen, Hendrik S{\"u}{\ss}]{Daniel H\"attig, J\"urgen Hausen, Hendrik S{\"u}{\ss}}

\address{Mathematisches Institut, Universit\"at T\"ubingen,
Auf der Morgenstelle 10, 72076 T\"ubingen, Germany}
\email{daniel.haettig@uni-tuebingen.de}

\address{Mathematisches Institut, Universit\"at T\"ubingen,
Auf der Morgenstelle 10, 72076 T\"ubingen, Germany}
\email{juergen.hausen@uni-tuebingen.de}

\address{Institut f\"ur Mathematik, Universit\"at Jena,
Ernst-Abbe-Platz 1-2, 07743 Jena, Germany}
\email{hendrik.suess@uni-jena.de}

\subjclass[2010]{14L30,14J26}

\begin{document}

\begin{abstract}
We consider two classes of non-toric log
del Pezzo $\CC^*$-surfaces: on the one side
the 1/3-log canonical ones and on the other
side those of Picard number one and Gorenstein
index at most 65.
In each of the two classes we figure
out the surfaces admitting a K\"ahler-Einstein
metric, a K\"ahler-Ricci soliton and those
allowing a Sasaki-Einstein metric
on the link of their anticanonical cone.
We encounter examples that 
admit a K\"{a}hler-Ricci soliton
but no Sasaki-Einstein cone link metric.
\end{abstract}

\maketitle

\section{Introduction}

A \emph{log del Pezzo surface} is a complex
projective surface that admits an ample
anticanonical divisor and has at most
singularities (locally) of the form $\CC^2/G$
with a finite subgroup $G \subseteq \GL_2(\CC)$.
The smooth del Pezzo surfaces are classically known:
these are $\PP_1 \times \PP_1$,
the projective plane $\PP_2$ and its blowing
up in up to $8$ points in general position.
The singular log del Pezzo surfaces comprise
infinitely many families and form an intensely studied
class in algebraic geometry; we just
mention~\cites{HiWa,AlNi,Nak,FuYa} representing
the state of the art concerning general
classification results.
Having at most finite quotient singularities,
the log del Pezzo surfaces are orbifolds and
thus also serve as an example class for
differential geometry.

A classical question is which of the log del Pezzo
surfaces allow a \emph{K\"ahler-Einstein metric},
that means a K\"ahler orbifold metric $g$ such
that the associated K\"ahler form~$\omega_g$ equals
its Ricci form $\mathrm{Ric}(\omega_g)$.
Concerning the smooth del Pezzos surfaces,
one finds K\"ahler-Einstein metrics on
$\PP_1 \times \PP_1$,
the projective plane $\PP_2$ and its blowing
up in $k = 3, \ldots, 8$ points in
general position 
as shown by Tian and Yau~\cites{TiYa,Ti1}.
The case of quasismooth del Pezzo
surfaces coming anticanonically embedded
into a three-dimensional weighted projective
space is understood as well due
to~\cites{KoJo,Arau,ChPaSh1,ChPaSh2}.
Dropping in the latter setting the assumption
of being anticanonically embedded or passing
to complete intersections instead of
considering just hypersurfaces
we arrive at open questions; see for
instance~\cites{KiPa,KiWo,LiuPe}.
Beyond the K\"ahler-Einstein metrics we have
the concept of a \emph{K\"ahler-Ricci soliton},
that means a K\"ahler metric $g$ together with
a vector field $\xi$ such that the
K\"ahler form~$\omega_g$ equals 
$\mathrm{Ric}(\omega_g)- L_\xi(\omega_g)$,
the difference of its Ricci form and its
Lie derivative with respect to $\xi$;
we refer to~\cite{CaSu} for the case of
Gorenstein log del Pezzo surfaces.

A related, also classical question concerns
the \emph{anticanonical link $S_X$}
of a log del Pezzo surface~$X$, that means the real
five-dimensional orbifold~$S_X$ given by
the boundary of a suitable neighbourhood of the apex
in the affine cone over $X$ defined by its anticanonical
divisor.
Here one asks if the anticanonical cone admits
a Ricci-flat K\"ahler metric, which then restricts
to a \emph{Sasaki-Einstein metric} on $S_X$, turning
it into a Sasaki-Einstein orbifold. 
Starting with the smooth K\"ahler-Einstein del Pezzo
surfaces, one obtains via this procedure precisely
the simply connected, regular, smooth
Sasaki-Einstein 5-folds, and,  all other
regular, smooth Sasaki-Einstein 5-folds arise as
cyclic quotients from the latter
ones~\cites{BFGK,FrKa}.
Beyond that but in the same spirit, the classification
of the simply connected
rational homology 5-spheres admitting a quasiregular
Sasaki-Einstein metric was recently completed by
detecting K\"ahler-Einstein metrics on certain singular
del Pezzo surfaces in a weighted projective
space~\cites{BoGa,BoNa,Kol1,Kol2,PaWo}.

In the present article, we consider log del Pezzo
surfaces $X$ that come with an effective action
of the multiplicative group $\CC^*$ and ask for
existence K\"ahler-Einstein metrics on $X$ and
Sasaki-Einstein metrics on the link $S_X$.
The case of toric log del Pezzo surfaces fits into
the larger class of toric Fano varieties and is
well understood~\cites{BoLe,CFO,FOW,Ma,ShiZhu,WaZh}:
As any toric Fano variety,
$X$ admits a K\"ahler-Einstein metric if and only
if the dual of its Fano polytope has the origin
as its barycenter.
Moreover, again as any toric Fano variety, $X$
admits a K\"ahler-Ricci soliton and a
Sasaki-Einstein metric on the link $S_X$.

We focus on the non-toric case,
that means that the acting~$\CC^*$ is a maximal
torus in the automorphism group of~$X$.
In this setting we benefit from the combinatorial
approaches~\cites{FlZa,AlHa,AlHaSu,HaSu,HaHe,HaWr}.
We give a brief introduction to the necessary background
in Section~\ref{sec:tvarlowcpl}
and then present explicit combinatorial criteria
for deciding existence of K\"ahler-Einstein metrics
or K\"ahler-Ricci solitons on~$X$, see
Theorems~\ref{thm:ke-crit} and~\ref{thm:krs-crit},
and a necessary criterion for Sasaki-Einstein metrics
on~$S_X$, see Theorem~\ref{thm:se-crit},
all allowing an efficient implementation.

The intention is to present a systematic view
on the numbers of families admitting
K\"ahler-Einstein metrics, K\"ahler-Ricci solitons
or possibly Sasaki-Einstein metrics on the link.
For this, we filter the (infinite) class of
families of non-toric log del Pezzo $\CC^*$-surfaces
by finite subclasses defined via suitable conditions
on the singularities.
In our first result, we consider $1/k$-log canonicity
and deliver for $k=1,2,3$, using the classification~\cite{HH}.
Recall that a variety is \emph{$\varepsilon$-log canonical}
if for any of its resolutions of singularities
the discrepancies are bounded from below by $\varepsilon - 1$.

\begin{theorem}
\label{thm:main1}
The numbers
$\mu_{\scriptscriptstyle \mathrm{KE}}(d)$,
$\mu_{\scriptscriptstyle \mathrm{KRS}}(d)$
$\mu_{\scriptscriptstyle \mathrm{SE}}(d)^*$
and $\mu(d)$
of $d$-dimensional families
of non-toric $1/3$-log canonical
del Pezzo $\CC^*$-surfaces
with a K\"ahler-Einstein metric,
a K\"ahler-Ricci soliton, 
possibly a Sasaki-Einstein metric on
the anticanonical link
and in total are given as

\medskip

\begin{center}
\begin{tabular}{c|c|c|c|c|c|c|c|c|c|c|c}
$d$  & $0$ & $1$ & $2$ & $3$ & $4$ & $5$ & $6$ & $7$ & $8$ & $9$ & $\ge 10$
\\
\hline
$\mu_{\scriptscriptstyle \mathrm{KE}}(d)$ & 227 & 177 & 56 & 33 & 12 & 9 & 4 & 3 & 1 & 1 & 0
\\
\hline
$\mu_{\scriptscriptstyle \mathrm{KRS}}(d)$ & 19695 & 3013 & 473 & 111 & 34 & 18 & 9 & 5 & 2 & 1 & 0
\\
\hline
$\mu_{\scriptscriptstyle \mathrm{SE}}(d)^*$ & 20230 & 3038 & 473 & 111 & 34 & 18 & 9 & 5 & 2 & 1 & 0
\\
\hline
$\mu(d)$ & 65022 & 12402 & 3190 & 917 & 254 & 64 & 14 & 6 & 2 & 1 & 0
\end{tabular}
\end{center}

\medskip
\noindent
Altogether, among the 81872 families
of non-toric $1/3$-log canonical del Pezzo
$\CC^*$-surfaces, we find 523 admitting a 
K\"ahler-Einstein metric, 23361 a  K\"ahler-Ricci
soliton and 23921 that possibly admit a
Sasaki-Einstein metric on the anticanonical link.
\end{theorem}

In our second result we consider the
non-toric log del Pezzo $\CC^*$-surfaces of
Picard number one and filter by the
Gorenstein index, as done in the
classification~\cite{HHHS}.
Recall that the \emph{Gorenstein index}
is the smallest positive integer $\iota$ such
that a $\iota$-fold canonical divisor is
Cartier.
The overlap with the class treated in the first
theorem consists of 253 families, 246 of them
sporadic and 7 of dimension one.

\goodbreak

\begin{theorem}
\label{thm:main2}
The numbers
$\nu_{\scriptscriptstyle \mathrm{KE}}(d)$,
$\mu_{\scriptscriptstyle \mathrm{KRS}}(d)$,
$\nu_{\scriptscriptstyle \mathrm{SE}}(d)^*$
and $\nu(d)$
of $d$-dimensional families
of non-toric log del Pezzo
$\CC^*$-surfaces of Picard number one
and Gorenstein index at most 65
with a K\"ahler-Einstein metric,
a K\"ahler-Ricci soliton,
possibly a Sasaki-Einstein metric on
the anticanonical link 
and in total are given~as

\medskip

\begin{center}
\begin{tabular}{c|c|c|c}
$d$  & $0$ & $1$ & $\ge 2$
\\
\hline
$\nu_{\scriptscriptstyle \mathrm{KE}}(d)$ & 255 & 37 & 0
\\
\hline
$\nu_{\scriptscriptstyle \mathrm{KRS}}(d)$ & 8167 & 443 & 0
\\
\hline
$\nu_{\scriptscriptstyle \mathrm{SE}}(d)^*$ & 8302 & 443 & 0
\\
\hline
$\nu(d)$ & 19301 & 443 & 0
\end{tabular}
\end{center}




\medskip
\noindent
Altogether, among the 19744
families of non-toric log del Pezzo
$\CC^*$-surfaces of Picard number one
and Gorenstein index at most 65,
there are 292 admitting a  K\"ahler-Einstein
metric, 8610 a K\"ahler-Ricci soliton
and 8745 that possibly admit a Sasaki-Einstein
metric on the anticanonical link.
Filtering by the Gorenstein index yields
the following picture:
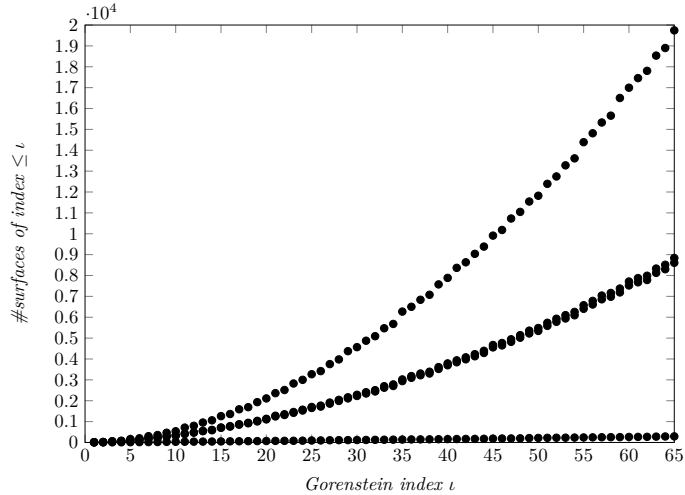
\begin{center}

\begin{tikzpicture}[scale=0.7]
\begin{axis}[
    xmin = 0, xmax = 65,
    ymin = 0, ymax = 20000,
    width = \textwidth,
    height = 0.75\textwidth,
    xtick distance = 5,
    ytick distance = 1000,
    xlabel = {Gorenstein index $\iota$},
    ylabel = {\#surfaces of index $\le \iota$}
]
\addplot[only marks] table {gi-le-65-sums.txt};
\end{axis}
\end{tikzpicture}

\end{center}

\end{theorem}

The log del Pezzo $\CC^*$-surfaces from Theorems~\ref{thm:main1}
and~\ref{thm:main2} together with their properties presented
there (and more) are available via the database~\cite{HHS}.

Our approach to metrics and solitons runs via equivariant
\emph{$K$-polystability}, which in turn is defined in terms
of equivariant \emph{test configurations}; see
Sections~\ref{sec:testconf} and~\ref{sec:polystability}.
Propositions~\ref{prop:testconf-APS}
and~\ref{prop:all-test-conf}
extend results of~\cite{IlSu} on equivariant test
configurations from the projective Gorenstein canonical
case to the semi-projective setting without any
conditions on the singularities beyond normality.
This allows us to obtain the combinatorial criteria 
given in Theorems~\ref{thm:ke-crit},~\ref{thm:krs-crit}
and~\ref{thm:se-crit} mentioned above in an analogous
manner to~\cite[Thm.~4.12]{IlSu}.

\begin{remark}
\label{rem:se-nec}
The numbers~$\mu_{\scriptscriptstyle \mathrm{SE}}^*(d)$
and $\nu_{\scriptscriptstyle \mathrm{SE}}^*(d)$ in
Theorems~\ref{thm:main1} and~\ref{thm:main2}
count in fact the families having a $K$-polystable
anticanonical cone and thus provide 
upper bounds for the numbers of $d$-dimensional
families of del Pezzo $\KK^*$-surfaces with a
Sasaki-Einstein metric on the anticanonical link.
\end{remark}

Finally, we look at the relations between
existence of a K\"ahler-Ricci soliton on
a log del Pezzo $\CC^*$-surface $X$ and
existence of a Sasaki-Einstein metric on
its anticanonical link.
For smooth Fano varieties it is conjectured that
existence of a K\"ahler-Ricci soliton implies
existence of a Sasaki-Einstein metric on the
anticanonical link~\cite{MaNa}.
Whereas any non-toric
del Pezzo $\KK^*$-surfaces of Picard number
one and Gorenstein index at most 65 with a
K\"ahler-Ricci soliton admits a Sasaki-Einstein
metric on the anticanonical link,
we find in the higher Picard numbers among
the 1/3 log canonical ones 11 families
with K\"ahler-Ricci soliton admitting
no Sasaki-Einstein metric on the anticanonical
link.
For instance, the following weighted projective
surface is a counterexample.

\begin{theorem}
\label{thm:counterex}
$V(
z_4^2-z_2z_5,
\, z_1^2z_2^2z_4 + z_2^3z_4 + z_3z_5,
\, z_1^2z_2^3 + z_2^4  + z_3z_4
)
\subseteq
\PP_{1,2,3,5,8}$
is a log del Pezzo surface admitting
a K\"ahler-Ricci soliton but no 
Sasaki-Einstein metric on its
anticanonical link.
\end{theorem}

The surface from Theorem~\ref{thm:counterex} is 
a non-toric $\CC^*$-surface of Picard number two.
It serves also as a running example throughout the text,
illustrating the summary of the combinatorial
approach~\cites{HaHe,HaSu} to rational $\CC^*$-surfaces
presented in Section~\ref{sec:tvarlowcpl}
and the construction of toric degenerations provided in
Section~\ref{sec:constr-tordeg}.
The proof of Theorem~\ref{thm:counterex} is given by
Examples~\ref{example:lbsp7} and~\ref{example:lbsp8}.

\medskip

We would like to thank Chi Li for valuable hints
and Justus Springer for patiently helping us
with computations and data management.

\goodbreak

\tableofcontents


\section{Torus actions of complexity one}
\label{sec:tvarlowcpl}

The~\emph{complexity} of the action of
an algebraic torus on a variety is the minimum
of the codimensions of its orbits.
This section summarizes the necessary background
on torus actions of complexity one.
We assume the reader to be familiar with
the basics on \emph{toric varieties}, that means
the case of complexity zero;
we refer to~\cites{Dan,Ful,CoLiSc} for
background.
We take the field $\CC$ of complex numbers as
base field and just remark that everything works
as well over any algebraically closed field of
characteristic zero.

We will make frequently use of Cox's quotient
presentation of a toric variety. Let us
briefly recall this construction and thereby
fix the corresponding notation.
Recall from~\cite{ArDeHaLa} that the
\emph{Cox ring} of a normal variety $X$ with
finitely generated divisor class group $\Cl(X)$
is the $\Cl(X)$-graded $\CC$-algebra
$$
\mathcal{R}(X)
\ := \
\bigoplus_{\Cl(X)} H^0(X, \mathcal{O}_D).
$$

\begin{construction}
\label{constr:coxtoric}
Let $Z$ be the toric variety defined by 
a fan  $\Sigma$ in a lattice $N$ such that
the primitive generators $v_1,\ldots, v_r$ 
of the rays of $\Sigma$ span the 
rational vector space $N_\QQ = N \otimes_\ZZ \QQ$.
We have a linear map
$$
P \colon \ZZ^r \ \to \ N,
\qquad
e_i \ \mapsto \ v_i.
$$
In the case $N = \ZZ^n$, we also speak
of the \emph{generator matrix}
$P = [v_1, \ldots, v_r]$ of $\Sigma$.
The divisor class group and the Cox ring 
of $Z$ are given by
$$
\Cl(Z) = K := \ZZ^r / \im(P^*),
\qquad
\mathcal{R}(Z) = \CC[T_1,\ldots,T_r],
\qquad
\deg(T_i) = Q(e_i),
$$
where $P^*$ denotes the dual map of $P$ and
$Q \colon \ZZ^r \to K$ the projection.
Finally, we obtain a fan $\hat \Sigma$ in $\ZZ^r$
consisting of certain faces of the positive orthant,
namely 
$$ 
\hat \Sigma 
\ := \ 
\{\delta_0 \preceq \QQ_{\ge 0}^r; \ 
P(\delta_0) \subseteq \sigma \text{ for some } \sigma \in \Sigma\}.
$$
The toric variety $\hat Z$ associated with $\hat \Sigma$
is an open toric subset in $\bar Z := \CC^r$.
As $P$ is a map of the fans~$\hat \Sigma$ and 
$\Sigma$, it defines a toric morphism 
$p \colon \hat Z \to Z$, which turns out to be
a good quotient with respect to the action
of the quasitorus 
$$
H \ = \ \Spec \, \CC[K] \ = \ \ker(p) \ \subseteq \ \TT^r.
$$
\end{construction}

We turn to torus actions of complexity
one. Our approach is the one provided
by~\cites{HaHe,HaSu} for the case of
rational projective varieties~$X$;
see also~\cite[Sec.~3.4]{ArDeHaLa}.
In fact, we need the slightly more general
framework~\cite[Constr.~1.6, Type~2]{HaWr}
covering for instance the 
semiprojective~$X$ with only constant
invariant functions,
where \emph{semiprojective}
means that $X$ admits a projective
morphism onto an affine variety.
In a first step one explicitly builds up the
prospective Cox rings of the
varieties in question from a pair $(A,P)$ of
defining matrices.

\begin{construction}
\label{constr:RAP}
Fix $r \in \ZZ_{\ge 1}$, a sequence 
$n_0, \ldots, n_r \in \ZZ_{\ge 1}$, set 
$n := n_0 + \ldots + n_r$, and fix  
integers $m \in \ZZ_{\ge 0}$ and $0 < s < n+m-r$.
The input data are matrices 
$$
A 
 =  
[a_0, \ldots, a_r]
 \in  
\Mat(2,r+1;\CC),
\qquad
P
 = 
\left[ 
\begin{array}{cc}
L & 0 
\\
d & d'  
\end{array}
\right]
 \in  
\Mat(r+s,n+m; \ZZ),
$$
where $A$ has pairwise linearly independent 
columns and $P$ is built from an
$(s \times n)$-block $d$, an $(s \times m)$-block 
$d'$ and an $(r \times n)$-block $L$ 
of the form  
$$
L
\ = \ 
\left[
\begin{array}{cccc}
-l_0 & l_1 &   \ldots & 0 
\\
\vdots & \vdots   & \ddots & \vdots
\\
-l_0 & 0 &\ldots  & l_{r} 
\end{array}
\right],
\qquad
l_i \ = \ (l_{i1}, \ldots, l_{in_i}) \ \in \ \ZZ_{\ge 1}^{n_i}
$$
such that the columns $v_{ij}$, $v_k$ of 
$P$ are pairwise different primitive vectors 
generating $\QQ^{r+s}$ as a vector space. 
Consider the polynomial algebra 
$$
\CC[T_{ij},S_k]
\ := \ 
\CC[T_{ij},S_k; \; 0 \le i \le r, \, 1 \le j \le n_i, 1 \le k \le m]. 
$$
Denote by $\mathfrak{I}$ the set of 
all triples $I = (i_1,i_2,i_3)$ with 
$0 \le i_1 < i_2 < i_3 \le r$ 
and define for any $I \in \mathfrak{I}$ 
a trinomial 
$$
g_I
\ := \
g_{i_1,i_2,i_3}
\ := \
\det
\left[
\begin{array}{ccc}
T_{i_1}^{l_{i_1}} & T_{i_2}^{l_{i_2}} & T_{i_3}^{l_{i_3}}
\\
a_{i_1} & a_{i_2} & a_{i_3}
\end{array}
\right],
\qquad
T_i^{l_i} 
\  := \
T_{i1}^{l_{i1}} \cdots T_{in_i}^{l_{in_i}}.
$$
Consider the factor group 
$K := \ZZ^{n+m}/\rm{im}(P^*)$
and the projection $Q \colon \ZZ^{n+m} \to K$.
We define a $K$-grading on 
$\CC[T_{ij},S_k]$ by setting
$$ 
\deg(T_{ij}) 
\ := \ 
w_{ij}
\ := \ 
Q(e_{ij}),
\qquad
\deg(S_{k}) 
 \ := \
w_k
\ := \  
Q(e_{k}).
$$
Then the trinomials $g_I$ just introduced 
are $K$-homogeneous, all of the same degree.
In particular, we obtain a $K$-graded 
factor algebra  
$$
R(A,P)
\ := \
\CC[T_{ij},S_k] 
\ / \
\bangle{g_I; \ I \in \mathfrak{I}}.
$$
\end{construction}

\begin{remark}
The ring $R(A,P)$ provided by Construction~\ref{constr:RAP}
is a normal complete intersection ring and its ideal of 
relations is generated by the $r-1$ trinomials 
$$
g_\iota
\ := \ 
g_{0,1,\iota},
\qquad
\iota \ = \ 2, \ldots, r.
$$
\end{remark}

Let us illustrate by means of a concrete example.
The example will accompany us throughout the whole
text and is the log del Pezzo $\CC^*$-surface from
Theorem~\ref{thm:counterex}.

\begin{example}
\label{example:lbsp1}
Let $r=2$, $n_0=n_1=2$, $n_2=1$ and thus
$n=5$. Moreover, let $m=0$, $s=1$ and consider 
the matrices
$$
A
\ := \
\left[
\begin{array}{rrr}
1 & 0 & -1 
\\
0 & 1 & -1
\end{array}
\right],
\qquad\qquad
P
\ := \
\left[
\begin{array}{rrrrr}
-2 & -1 & 1 & 1 & 0
\\
-2 & -1 & 0 & 0 & 2
\\
3  & -1 & 0 & -1 & 1
\end{array}
\right].
$$
These data fit into the scheme of
Construction~\ref{constr:RAP} with 
$L$ consisting of the upper two rows
of $P$. We directly see
$$
R(A,P)
\ = \
\CC[T_{01},T_{02},T_{11},T_{12},T_{21}]
/
\bangle{T_{01}^2T_{02} + T_{11}T_{12} + T_{21}^2}.
$$
Moreover, we obtain $K = \ZZ^5 / \im(P^*) = \ZZ^2$
with the projection $\ZZ^5 \to \ZZ^2$ given by the
degree matrix
$$
Q
\ = \
[w_{01},w_{02},w_{11},w_{12},w_{21}]
\ = \ 
\left[
\begin{array}{rrrrr}
0 & 2 & 3 & -1 & 1
\\
1 & 2 & 1 &  3 & 2 
\end{array}
\right].
$$
\end{example}

In the second step, our varieties $X$ with torus
action of complexity one are constructed as quotients
of $\bar X = \Spec \, R(A,P)$ by the quasitorus 
$H = \Spec \, \CC[K]$. Each of these $X$ comes 
embedded into a toric variety.

\begin{construction}
\label{constr:RAPdown}
Situation as in Construction~\ref{constr:RAP}.
Consider the common zero set of the 
defining relations of $R(A,P)$:
$$ 
\bar{X} 
\ := \ 
V(g_I; \ I \in \mathfrak{I}) 
\ \subseteq \ 
\bar{Z} 
\ := \ 
\CC^{n+m}
$$
Let $\Sigma$ be any fan in the lattice $N = \ZZ^{r+s}$
having the columns of $P$ as the primitive 
generators of its rays.
Construction~\ref{constr:coxtoric} leads to 
a commutative diagram
$$ 
\xymatrix{
{\bar{X}}
\ar@{}[r]|\subseteq
\ar@{}[d]|{\rotatebox[origin=c]{90}{$\scriptstyle\subseteq$}}
&
{\bar{Z}}
\ar@{}[d]|{\rotatebox[origin=c]{90}{$\scriptstyle\subseteq$}}
\\
{\hat{X}}
\ar[r]
\ar[d]_{\quot H}^p
&
{\hat{Z}}
\ar[d]^{\quot H}_p
\\
X
\ar[r]
&
Z}
$$
with a normal closed subvariety
$X = X(A,P,\Sigma) := p(\hat X)$
of the toric variety $Z$ associated with~$\Sigma$.
Dimension, divisor class group
and Cox ring of $X$ are given~by 
$$ 
\dim(X) = s+1,
\qquad
\Cl(X) \ \cong \ K,
\qquad
\mathcal{R}(X) \ \cong \ R(A,P).
$$
The subtorus $T \subseteq \TT^{r+s}$ of the 
acting torus of $Z$ associated with the  
sublattice $\ZZ^{s} \subseteq \ZZ^{r+s}$
leaves $X$ invariant and the induced $T$-action
on $X$ is of complexity one.
\end{construction}

\begin{theorem}
See~\cites{HaSu,HaHe,HaWr}.
Every normal semiprojective rational variety with
a torus action of complexity one
having only constant invariant functions
is equivariantly isomorphic to some $X(A,P,\Sigma)$.
\end{theorem}

\begin{example}
\label{example:lbsp2}
In the setting of Example~\ref{example:lbsp1},
we have $\bar X = V(T_{01}^2T_{02} + T_{11}T_{12} + T_{21}^2)$,
embedded into $\bar Z = \CC^5$.
Take $\Sigma$ to be the fan in $\ZZ^5$
having
$$
\cone(v_{01},v_{11},v_{21}), \quad
\cone(v_{02},v_{12},v_{21}), \quad
\cone(v_{01},v_{02}), \quad
\cone(v_{11},v_{12})
$$
as its maximal cones. Then $\hat Z$ is the open
subset of $\bar Z$ obtained by removing the
zero sets of the monomials
$$
T_{02}T_{12}, \qquad
T_{01}T_{11}, \qquad
T_{11}T_{12}T_{21}, \qquad
T_{01}T_{02}T_{12}.
$$
The quotient map $p \colon \hat Z \to Z$ by the
action of $H = \ker(p)$ is given
on the respective acting tori $\TT^5 \subseteq \hat Z$
and $\TT^3 \subseteq Z$ as
$$
p \colon \TT^5 \ \to \ \TT^3,
\qquad
(t_{01}, \, t_{02},  \, t_{11},  \, t_{12},  \, t_{21})
\ \mapsto \
\left(
\frac{t_{11}t_{12}}{t_{01}^2t_{02}}, \
\frac{t_{21}^2}{t_{01}^2t_{02}}, \
\frac{t_{01}^3t_{21}}{t_{02}t_{12}}
\right).
$$
Thus, for $X = p(\hat X)$, we obtain
$X \cap \TT^3 = V(1+t_1+t_2)$.
In particular, we see that~$X$ is invariant
under the subtorus
$\{(1,1,t); \ t \in \CC^*\}$ of $\TT^3$.
Moreover, we have
$$ 
\dim(X) = s+1 = 2,
\qquad
\Cl(X) \ \cong \ K \ = \ \ZZ^2,
\qquad
\mathcal{R}(X) \ \cong \ R(A,P).
$$
\end{example}


\begin{remark}
Construction~\ref{constr:RAPdown}
provides us with \emph{Cox coordinates}.
Each $z \in Z$ can be written as
$z = p(\hat z)$, where $\hat z \in \hat Z \subseteq \CC^{n+m}$ 
has closed $H$-orbit in $\hat Z$ and this 
presentation is unique up to multiplication 
by elements of $H$.
This allows us in particular to put hands
on the points of $X = X(A,P,\Sigma) \subseteq Z$.
\end{remark}

\begin{remark}
We say that the matrix $P$ from Construction~\ref{constr:RAP}
is \emph{irredundant} if we have $l_{i1}n_i \ge 2$ for
$i = 0, \ldots, r$.
In Construction~\ref{constr:RAPdown}, we may assume
without loss of generality that $P$ is irredundant.
An $X(A,P,\Sigma)$ with irredundant $P$ is a toric
variety if and only if $r = 1$ holds.
\end{remark}

\begin{remark}
Consider $X = X(A,P,\Sigma)$ and the embedding
$X \subseteq Z$ as provided by 
Construction~\ref{constr:RAPdown}.
The toric prime divisors
$D_{ij}^Z, D_k^Z \subseteq Z$
defined by the columns of $P$
restrict
to $\TT^s$-invariant prime divisors
$D_{ij}^X, D_k^X \subseteq X$
such that
$$
[D_{ij}^X]
\ = \
[D_{ij}^Z],
\qquad\qquad
[D_{ij}^X]
\ = \
[D_{ij}^Z]
$$
holds for the respective classes in $\Cl(X) = K = \Cl(Z)$.
Moreover, the projection $\ZZ^{r+s} \to \ZZ^r$ defines
the rational quotients 
$\pi_Z \colon Z \dasharrow \PP_r$
for the action of $\TT^s$.
Restricting yields the rational quotient 
$\pi_X \colon X \dasharrow \PP_1$
and a commutative diagram
$$ 
\xymatrix{
  X
\ar@{}[r]|\subseteq
\ar@{-->}[d]_{\pi_X}
&
Z
\ar@{-->}[d]^{\pi_Z}
\\
{\PP_1}
\ar@{}[r]|\subseteq
&
{\PP_r}
}
$$
where $\pi_Z$ is defined on the union $Z' \subseteq Z$
of the toric orbits of codimension at most one and
$\pi_X$ maps $X' = X \cap Z'$ onto $\PP_1$.
Moreover, $\PP_1$ intersects the
$i$-th coordinate hyperplane of $\PP_r$
in the point
$\alpha_i = [a_{i1},a_{i2}] \in \PP_1$ given by the
$i$-th column of $A$ and
$$
\overline{\pi_X^{-1}(\alpha_i)}
\ = \ 
D_{i0}^X \cup \ldots \cup D_{in_i}^X \subseteq \ X
$$
holds for $i = 0, \ldots, r$. If $P$ is irredundant,
then the $\pi_X^{-1}(\alpha_i)$ are precisely the
critical fibres of $\pi_X$.
If the situation is clear, then we will simply
write $D_{ij}$ and $D_k$ for the divisors on
$X$ and as well on $Z$.
\end{remark}

\begin{remark}
\label{rem:families}
Let $X = X(A,P,\Sigma)$ with $P$ irredundant.
Varying $A$ yields different constellations
of the critical values of
$\pi_X \colon X \dasharrow \PP_1$, which is the
reason for moduli.
More precisely, we can always assume
$$
A
\ = \
\left[
\begin{array}{cccccc}
1 & 0 & -\lambda_1 & -\lambda_2 & \cdots & -\lambda_r
\\
0 & 1 & -1 & -1 & \cdots & -1
\end{array}
\right],
\qquad
g_\iota \ = \ \lambda_\iota T_0^{l_0} + T_1^{l_1} + T_\iota^{l_\iota}
$$
with $\lambda_1 = 1$, all $\lambda_\iota$ non-zero and
pairwise distinct.
As a consequence, we see that for fixed $P$ and $\Sigma$,
the $X(A,P,\Sigma)$ come in an $(r-2)$-dimensional family,
parametrized by the $\lambda_i$
(up to possible renumberings).
\end{remark}

\begin{remark}
Let $X = X(A,P,\Sigma)$ and $X \subseteq Z$ 
be as in Construction~\ref{constr:RAPdown}.
Then, in $K_\QQ = \Cl_\QQ(X)$, the cones of
effective and movable divisor classes are
given as
$$
\Eff(X) \ = \ \Eff(X) \ = \ Q(\gamma),
\qquad\qquad
\Mov(X) \ = \ \Mov(Z) \ = \
\bigcap_{\genfrac{}{}{0pt}{1}{\gamma_0 \preccurlyeq \gamma}{\scriptstyle\rm facet}} Q(\gamma_0),
$$
where $\gamma = \cone(e_{ij},e_k) \subseteq \QQ^{n+m}$
denotes the positive orthant.
Moreover, the cones of semiample divisor classes
of $X$ and $Z$ are related to each other as
follows:
$$
\SAmple(Z)
\ = \
\bigcap_{\sigma \in \Sigma} \sigma^*
\ \subseteq \
\bigcap_{\sigma \in \Sigma_X} \sigma^*
\ = \
\SAmple(X),
$$
where $\sigma^* = \cone(w_{ij},w_k; \ v_{ij},v_k \not\in \sigma) \subseteq K_\QQ$
is the complementary cone and
$\Sigma_X$ consists of all $\sigma \in \Sigma$
such that the corresponding toric orbit meets $X \subseteq Z$.
\end{remark}

\begin{remark}
\label{rem:projemb}
Let $X = X(A,P,\Sigma) \subseteq Z$ 
be as in Construction~\ref{constr:RAPdown}.
The cone of ample divisor classes is
the relative interior of the cone of semiample
divisor classes:
$$
\Ample(X) \ = \ \SAmple(X)^\circ.
$$
Any $w \in \Ample(X)$ gives rise to an equivariant
open embedding $X \subseteq X_w$, where
$X_w = X(A,P,\Sigma_w) \subseteq Z_w$ and
the fan $\Sigma_w$ in $\ZZ^{r+s}$ is given by
$$
\Sigma_w
\ := \
\{P(\gamma_0); \ \gamma_0 \preccurlyeq \gamma, \ w \in Q(\gamma_0)^\circ\}.
$$
We have $X = X_w$ if and only if $X \subseteq Z_w$
meets the same toric orbits as $X_w \subseteq Z_w$.
If so, then $X$ is semiprojective and it is
projective if in addition $\QQ^{r+s} = \cone(v_{ij},v_k)$.
\end{remark}

\begin{example}
\label{example:lbsp3}
We continue Examples~\ref{example:lbsp1}
and~\ref{example:lbsp2}.
In $\Cl_\QQ(X)=K_\QQ=\QQ^2$, the
columns of the degree matrix
$Q = [w_{01},w_{02},w_{11},w_{12},w_{21}]$ are
located as follows.

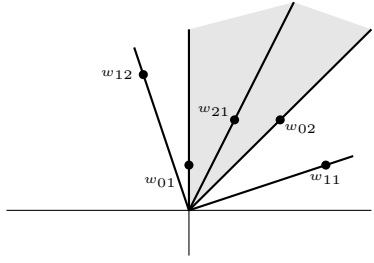
\begin{center}
\begin{tikzpicture}[scale=0.6]
\sffamily
\coordinate(oo) at (0,0);
\coordinate(w01) at (0,1);
\coordinate(w02) at (2,2);
\coordinate(w11) at (3,1);
\coordinate(w12) at (-1,3);
\coordinate(w21) at (1,2);
\coordinate(c01) at (0,4);
\coordinate(c02) at (4,4);
\coordinate(c11) at (3.6,1.2);
\coordinate(c12) at (-1.2,3.6);
\coordinate(c21) at (2.3,4.6);
\path[fill, color=gray!20] (oo) -- (c02) -- (c21) -- (c01) -- (oo);
%
\path[fill, color=black] (w01) circle (1mm) 
node[below left=.5mm]{$\scriptscriptstyle w_{01}$};
\path[fill, color=black] (w02) circle (1mm) 
node[below right=-1mm]{$\scriptscriptstyle w_{02}$};
\path[fill, color=black] (w11) circle (1mm)
node[below]{$\scriptscriptstyle w_{11}$};
\path[fill, color=black] (w12) circle (1mm)
node[left]{$\scriptscriptstyle w_{12}$};
\path[fill, color=black] (w21) circle (1mm)
node[above left=-1mm]{$\scriptscriptstyle w_{21}$};
\draw[thick,color=black] (oo) to (c01);
\draw[thick,color=black] (oo) to (c02);
\draw[thick,color=black] (oo) to (c11);
\draw[thick,color=black] (oo) to (c12);
\draw[thick,color=black] (oo) to (c21);
\draw[color=black] (-4,0) to (4,0);
\draw[color=black] (0,-1) to (0,3);
\end{tikzpicture}
\end{center}

\noindent
The shadowed part indicates $\Mov(X)$, which
equals $\SAmple(X)$ as we are in the surface
case.
Remark~\ref{rem:projemb} yields three distinct
equivariant closed embeddings
$$
X \ \subseteq \ Z_{w_{02} + w_{21}},
\qquad
X \ \subseteq \ Z_{w_{21}},
\qquad
X \ \subseteq \ Z_{w_{21} + w_{01}},
$$
showing in particular that $X$ is projective.
We take a closer look at $w_{21} = [D_{21}^X]$
and the associated graded ring
$$
R(w_{21})
\ := \ 
\bigoplus_{k \geq 0} H^0(X,\mathcal{O}(kD_{21}^X))
\ = \
\bigoplus_{k \ge 0} R(A,P)_{kw_{21}}
\ \subseteq \
R(A,P).
$$
Being a Veronese subalgebra of $R(A,P)$, we
can compute for $R(w_{21})$ an explicit presentation
in terms of homogeneous generators and relations:
$$
R(w_{21})
\cong
\CC[T_1, \ldots, T_5]
/
\bangle{ T_4^2-T_2T_5, \, T_1^2T_2^2T_4 + T_2^3T_4 + T_3T_5, \, T_1^2T_2^3 + T_2^4  + T_3T_4 },
$$
where $\CC[T_1, \ldots, T_5]$ is graded by
assigning to $T_1,T_2,T_3,T_4,T_5$ the degrees
$1,2,3,5,8$.
We conclude that $X$ is the weighted
projective surface from Theorem~\ref{thm:counterex}.
\end{example}

\begin{remark}
\label{rem:antican}
Let $X = X(A,P,\Sigma)$ and $X \subseteq Z$ 
be as in Construction~\ref{constr:RAPdown}.
Then the anticanonical class of $X$ in
$\Cl(X) = K$ is given by 
$$
[-\mathcal{K}_X]
\ = \
 (1-r)\mu + \sum w_{ij} + \sum w_k,
\quad
\mu \ = \
\sum_{j=1}^{n_0} l_{0j} w_{0j}
\ = \
\ldots
\ = \
\sum_{j=1}^{n_r} l_{rj} w_{rj}.
$$
This allows us in particular to decide if $X$ is a Fano variety,
that means if  $[-\mathcal{K}_X]$ lies
in the ample cone
$\Ample(X) = \SAmple(X)^\circ$.
\end{remark}

\begin{example}
\label{example:lbsp4}
The $X = X(A,P,\Sigma)$ provided by Examples~\ref{example:lbsp1}
and~\ref{example:lbsp2} is a del Pezzo surface, as we have
$$
[-\mathcal{K}_X]
\ = \
w_{01} + w_{02} + w_{21}
\ = \
(5,3)
\ \in \
\Ample(X)
\ = \
\cone(w_{02},w_{01})^\circ.
$$
\end{example}

Given a projective variety with a torus action
and an ample divisor,
we can lift the action via linearization to
the corresponding affine cone and, together
with the fiberwise $\CC^*$-action, this gives
a torus action of the same complexity.
For complexity at most one, such equivariant
cone constructions are obtained via combinatorial
data as follows; note that for $r=1$, the
statement settles the toric case.

\goodbreak

\begin{proposition}
\label{prop:affinecone}
Let $X \subseteq Z$ arise from $(A,P,\Sigma)$ with
$\Sigma$ polytopal and consider ample divisors
$D^Z$ on $Z$ and $D^X$ on $X$ given by
$$
D^Z
\ = \
\sum_{i,j} \alpha_{ij} D^Z_{ij} + \sum_{k} \alpha_k D^Z_k,
\qquad
D^X
\ = \
\sum_{i,j} \alpha_{ij} D^X_{ij} + \sum_{k} \alpha_k D^X_k.
$$
Then the associated affine cones~$X' \subseteq Z'$
over $X \subseteq Z$
arise via Construction~\ref{constr:RAPdown}
from $(A,P',\Sigma')$ with $s' := s+1$ and 
the $(r+s')$ by $(n+m)$ stack matrix
$$
P'
\ := \
\left[
\begin{array}{c}
P
\\
\alpha
\end{array}
\right],
\qquad
\alpha
\ = \
(\alpha_{ij},\alpha_k)
\ \in \ \ZZ^{n+m}
$$
and $\Sigma'$ is the fan of faces of the cone
$\sigma' \subseteq \QQ^{r+s'}$
generated by the columns of $P'$.
With the projection
$F \colon \ZZ^{r+s'} \to \ZZ^{r+s}$
onto the first $r+s$ coordinates
we have 
\[
\Sigma
\ = \   
\{F(\sigma_0'); \ \sigma_0' \prec \sigma'\}
\]
and the cone projection
$Z' \setminus \{z'\} \to Z$
restricts to the cone projection
$X' \setminus \{x'\} \to X$,
where $x'=z'$ is the common apex
of the affine cones $X' \subseteq Z'$.
\end{proposition}

\begin{proof}
Clearly, $(A,P',\Sigma')$ are defining data in the sense of
Construction~\ref{constr:RAPdown}.
Moreover, by the definition of the stack matrix $P'$, we have
\[
K'
\ = \
\ZZ^{r+s'}/\im((P')^*)
\ = \
(\ZZ^{r+s}/\im(P^*)) / \ZZ w_D
\ = \
K / \ZZ w_D,
\]
where we write $w_D \in K$ for the common class of
$D^Z$ and $D^X$ in $\Cl(Z) = \Cl(X)$.
This allows us to identify the affine cone $X'$
via its algebra of functions:
\[
\Gamma(X',\mathcal{O})
\ \cong \
R(A,P')_0
\ = \
\bigoplus_{n \in \ZZ} R(A,P)_{nw_D}
\ \cong \
\bigoplus_{n \in \ZZ_{\ge 0}} \Gamma(X,\mathcal{O}(nD)).
\]
\end{proof}


\section{A normality lemma}

Given an affine variety $X = X(A,P)$, we consider
the acting torus $\TT^{r+s} = \TT^r \times \TT^s$
of its ambient toric variety $Z$ and the
subtorus $\TT = \TT^1 \times \TT^s$,
where $\TT^1 \subseteq \TT^r$ is the one-dimensional
diagonal and $\TT^s$ the acting torus of $X$.
We figure out implications of normality of the closure
of $\TT$ in $Z$ on the defining matrix $P$.
The result will be used later for characterizing
the normality of toric degenerations.

\begin{proposition}
\label{lem:loch}
Consider the affine toric variety $Z$ arising from the
lattice cone~$\sigma$ in~$\ZZ^{r+s}$ generated by
primitive vectors of the form
\[
v_i := l_ie_i + d_{i1} e_{r+1} + \ldots + d_{is} e_{r+s},
\quad
l_1 \ge \ldots \ge l_r \ge 1,
\quad
0 \le d_{ij} < l_i.
\]
and $\varphi \colon \TT^{1+s} \to \TT^{r+s}$,
$(t_1,t_2, \ldots, t_{1+s}) \mapsto (t_1, \ldots, t_1,t_2,\ldots,t_{1+s})$.
Assume that the closure of the image
$X := \overline{\varphi(\TT^{1+s})} \subseteq Z$ is normal.
Then $l_2 = \ldots = l_r = 1$.
\end{proposition}

\begin{proof}
We split the considerations into three steps.
The first two steps handle special cases and
the third step shows how to settle the
general case by suitably combining the first
two.

\medskip

\noindent
\emph{Step 1.}
Let $s=1$. Then $\Gamma(X,\mathcal{O}) = \varphi^*\Gamma(Z,\mathcal{O})$
is a subalgebra of $\Gamma(\TT^2,\mathcal{O}) = \CC[\ZZ^2]$, namely
the monoid algebra associated with
\[
F^*(\sigma^\vee \cap \ZZ^{r+1})
\ \subseteq \
\ZZ^2,
\qquad
F \ := \
\left[
\begin{array}{cc}
\eins_{r} & 0
\\      
0 & 1
\end{array}
\right],
\qquad
\eins_r \ := \ (1,\ldots,1) \ \in \ \ZZ^r.
\]
Here, the $r\times 2$ matrix $F$ stores the
exponents of the homomorphism
$\varphi \colon \TT^{2} \to \TT^{r+1}$.
Write $d_i := d_{i1}$ and 
set $\ell := l_1 \cdots l_r$ and $\ell_i := \ell/l_i$.
Then we obtain a primitive vector
\[
u
\ := \
\frac{1}{\gcd(\ell_1d_1, \ldots, \ell_rd_r,\ell)} \cdot
\left(
- \ell_1d_1 e_1 - \ldots - \ell_rd_r e_r + \ell e_{r+1}
\right)
\ \in \ \ZZ^{r+1}.
\]
Note that $u$ generates the lineality space of the dual
cone $\sigma^\vee \subseteq \QQ^{r+1}$.
Moreover, $\sigma^\vee$ is generated as a cone
by $\pm u$ together with $e_1,\ldots,e_r$.
Observe
\[
\ker(F^*) \ = \ \lin_\ZZ (e_i-e_1; \, i=2,\ldots,r),
\qquad\qquad
\ker(F^*) \cap \sigma^\vee \ = \ \{0\}.
\]
In particular, $F^*$ maps the lineality space of
$\sigma^\vee$ onto that of $F^*(\sigma^\vee)$.
Thus, the group of invertible elements
of $F^*(\sigma^\vee \cap \ZZ^{r+1})$
is generated by 
\[
F^*(u)
\ := \
\frac{1}{\gcd(\ell_1d_1, \ldots, \ell_rd_r,\ell)} \cdot
\left(
- (\ell_1d_1 + \ldots + \ell_rd_r) e_1 + \ell e_2
\right)
\ \in \ \ZZ^{2}.
\]
By normality of $X$, the monoid
$F^*(\sigma^\vee \cap \ZZ^{r+1}) \subseteq \ZZ^{2}$ is saturated
and thus the vector $F^*(u) \in \ZZ^{2}$ is primitive.
The latter gives us
\[
\gcd(\ell_1d_1 + \ldots + \ell_rd_r,\ell)
\ = \ 
\gcd(\ell_1d_1, \ldots, \ell_rd_r,\ell)
\ = \
\gcd(\ell_1, \ldots, \ell_r).
\]
Now write $F^*(u) = - \mu e_1 + \nu e_2$.
Then $\mu$ and $\nu$ are coprime integers
with $\mu \ge 0$ and $\nu > 0$.
Using the above identity, we see
\[
\nu
\ = \
\frac{\ell}{\gcd(\ell_1d_1, \ldots, \ell_rd_r,\ell)}
\ = \
\frac{\ell}{\gcd(\ell_1, \ldots, \ell_r)}
\ = \
\lcm(l_1, \ldots, l_r).
\]
Moreover, $F^*(\sigma^\vee) \subseteq \QQ^2$
is the half space of vectors that evaluate
non-negatively on $\nu e_1 + \mu e_2$.
Choose $\alpha \in \ZZ$ and
$\beta \in \ZZ > 0$ such that 
$\alpha \nu  + \beta \mu = 1$. Consider
\[
\tilde u
\ := \
\alpha e_1 + \beta e_2
\ \in \
F^*(\sigma^\vee) \cap \ZZ^{2}.
\]
Then $\tilde u$ and the $\pm F^*(u)$
form a generator system for the monoid
$F^*(\sigma^\vee) \cap \ZZ^{2}$.
Consequently, using again normality of $X$, we have
\[
\tilde u
\ \in \
F^*(\sigma^\vee \cap \ZZ^{r+1}).
\]

Now, fix an integer $1 \le i \le r$.
Then the elements of $\ZZ^{r+1}$ mapping via $F^*$
onto the vector $\tilde u \in \ZZ^2$ are all of the form
\[
\tilde u_i (\xi)
 = 
(\xi_i+\alpha) e_i + \sum_{\iota \ne i} \xi_\iota e_\iota + \beta e_{r+1},
\quad
\xi_1, \ldots, \xi_r \in \ZZ, \ \xi_1+ \ldots + \xi_r = 0.
\]
At least one of them satisfies $\tilde u_i (\xi) \in \sigma^\vee$
and thus evaluates non-negatively on each of the generators
$v_1,\ldots,v_r \in \sigma$. Concretely that means
\[
(\xi_i + \alpha)l_i + \beta d_i \ \ge \ 0,
\qquad\qquad
\xi_\iota l_\iota + \beta d_\iota  \ \ge  \ 0, \quad \iota \ne i.
\]
Using $\alpha = (1-\beta \mu)/\nu$
and $\mu/\nu = d_1/l_1 + \ldots + d_r/l_r$
together with
$\xi_i = - \sum_{\iota \ne i} \xi_\iota$,
we can bound the integer $\xi_i$
from below and above via $-\beta d_i/l_i$
as follows: 
\[
-\alpha - \beta \frac{d_i}{l_i}
\ = \ 
\beta \sum_{\iota \ne i} \frac{d_\iota}{l_\iota}   - \frac{1}{\nu}
\ \le \
\xi_i
\ \le \ 
\beta  \sum_{\iota \ne i} \frac{d_\iota}{l_\iota}
\ = \
-\alpha - \beta \frac{d_i}{l_i} + \frac{1}{\nu}.
\]

Write
$\beta d_i/l_i = b_i + \delta_i/l_i$ with
$b_i \in \ZZ$ and $0 \le \delta_i < l_i$.
Adding $\alpha + b_i$ to the inequalities,
we see that existence of a
$\tilde u_i(\xi) \in \sigma^\vee$
for $i = 1, \ldots, r$ is equivalent to
\[
-\frac{\delta_i}{l_i}
\ \le \
0
\ \le \
\frac{1}{\nu} - \frac{\delta_i}{l_i},
\qquad i = 1,\ldots, r.
\]
We claim $\delta_i=0$ if and only if $l_i=1$.
Clearly, $l_i=1$ implies $\delta_i=0$.
Conversely, $\delta_i=0$ forces $l_i \mid \beta$,
which due to $l_i \mid \nu$ and
$\alpha \nu + \beta \mu = 1$ amounts to
$l_i=1$.

Now, let $1 \le k \le r$ be maximal with
$\delta_k > 0$. Then $l_i=1$ and $\delta_i=0$
for $i>k$.
Moreover, $l_1 \ge \ldots \ge l_k > 1$ and
$\delta_i \nu \le l_i$.
Thus, $\nu = \lcm(l_1,\ldots,l_r)$ gives us
\[
\nu 
\ = \
l_1
\ = \
\ldots
\ = \
l_k
\ > \
1,
\qquad
\delta_1
\ = \
\ldots
\ = \
\delta_k
\ = \
1.
\]
Consequently, $\ell_i = \nu^{k-1}$ for $i=1,\ldots,k$
and $\mu = d_1+\ldots+d_k$, as $l_i=1$ implies $d_i=0$.
With $\beta d_i/l_i = b_i + \delta_i$ as
above, the Bezout identity
$\alpha \nu + \beta \mu = 1$
finally becomes
\[
\alpha
\ = \
\frac{1-\beta \mu}{\nu}
\ = \
\frac{1}{\nu}
- \beta \frac{d_1+\ldots+d_k}{\nu}
\ = \
(1-k) \frac{1}{\nu}
- b_1 - \ldots - b_k .
\]
For $\nu = 1$ or $k=1$,
we obtain the assertion.
We exclude the case $\nu,k \ge 2$.
There we have $k-1 = \gamma \nu$ with
$\gamma \in \ZZ_{\ge 1}$.
Setting $b := b_1 + \ldots + b_k$ one obtains
\[
\xi_i
 \ge 
- \alpha - \beta \frac{d_i}{\nu}
 = 
\gamma + b - b_i - \frac{1}{\nu},
\qquad
\xi_\iota
 \ge 
- \beta \frac{d_\iota}{\nu}
=
\begin{cases}
-b_\iota - \frac{1}{\nu}, & i \ne \iota = 1, \ldots , k,
\\
0, & \iota = k+1, \ldots, r
\end{cases}
\]
by evaluating $\tilde u_i(\xi)$ on $v_1, \ldots, v_r$.
Due to $\xi \in \ZZ^{r+1}$, this means
$\xi_i \ge \gamma + b - b_i$ and $\xi_\iota \ge -b_\iota$
for $i \ne \iota = 1, \ldots , k$.
Thus, $0 = \xi_1 + \ldots + \xi_r \ge \gamma > 0$;
a contradiction.

\medskip

\noindent
\emph{Step 2.}
Assume $s \in \ZZ_{\ge 2}$ and that there is a $1 \le p \le r$
such that $v_i = e_i$ holds for $i = p+1, \ldots, r$ and
with $d_i := d_{ii}$, the vectors $v_1, \ldots, v_p$ are of
the shape
\[
v_i  =  l_i e_i + d_i e_{r+i},
\qquad
0 < d_i < l_i,
\qquad
\gcd(l_i,d_i) = 1.
\]
We show that this forces $p=1$.
As in the first step,
$\Gamma(X,\mathcal{O}) = \varphi^*\Gamma(Z,\mathcal{O})$
is the subalgebra of $\Gamma(\TT^{1+s},\mathcal{O}) = \CC[\ZZ^{1+s}]$
given by the monoid algebra associated with 
\[
F^*(\sigma^\vee \cap \ZZ^{r+s})
\ \subseteq \
\ZZ^{1+s},
\qquad\qquad
F \ := \
\left[
\begin{array}{cc}
\eins_{r} & 0
\\      
0 & E_s
\end{array}
\right],
\]
where $\eins_{r} = (1,\ldots,1) \in \ZZ^r$
and this time we encounter the $s \times s$
unit matrix $E_s$ instead of just $1$.
Consider the primitive vectors 
\[
u_i \ := \ -d_i e_i + l_i e_{r+i} \ \in \ \ZZ^{r+s},
\quad i = 1 ,\ldots, p.
\]
Then $u_1, \ldots, u_p,e_{r+p+1}, \ldots, e_{r+s}$ generate
the lineality space of $\sigma^\vee \subseteq \QQ^{r+s}$ 
and the cone $\sigma^\vee$ is generated by its
lineality space together with $e_1, \ldots, e_r$.
As before,
\[
\ker(F^*) \ = \ \lin_\ZZ (e_i-e_1; \, i=2,\ldots,r),
\qquad\qquad
\ker(F^*) \cap \sigma^\vee \ = \ \{0\}.
\]
So, $F^*$ maps the lineality space of
$\sigma^\vee$ onto that of
$F^*(\sigma^\vee) \subseteq \QQ^{1+s}$.
Thus, the lineality space of $F^*(\sigma^\vee)$
is generated by the vectors 
$e_{p+1}, \ldots, e_s$ and
\[
F^*(u_i)
\  = \
-d_i e_1 + l_i e_{1+i}
\ \in \
\ZZ^{1+p}
\ \subseteq \ 
\ZZ^{1+s},
\qquad
i = 1 , \ldots , p.
\]
As a cone $\sigma^\vee$ is generated by 
its lineality space and the vector $e_1$.
We conclude that $\sigma^\vee \subseteq \QQ^{1+s}$
is the positive half space of the linear form
\[
\ell e_1 + \ell_1 d_1 e_2 + \ldots +  \ell_p d_p e_{1+p}
\ \in \
\ZZ^{1+s}.
\]
By normality of $X$, the $F^*(u_i)$ must generate
a primitive sublattice of $\ZZ^{1+s}$.
Thus, there is a vector
$\tilde u = \alpha e_1 + \beta_1 e_1 + \ldots + \beta_p e_p
\in \ZZ^{1+p} \subseteq \ZZ^{1+s}$
satisfying
\[
\alpha \ell + \beta_1 \ell_1 d_1 + \ldots + \beta_p \ell_p d_p 
\ = \   
\det(\tilde u,F^*(u_1),\ldots,F^*(u_p))
\ = \
1.
\]
Note that $\tilde u \in F^*(\sigma^\vee)$.
Again by normality of $X$, we have
$\tilde u \in F^*(\sigma^\vee \cap \ZZ^{r+s})$.
For fixed $i$, the elements of $\ZZ^{r+s}$ mapping
via $F^*$ onto $\tilde u$ are all of the form
\[
\tilde u_i(\xi)
\ = \
\alpha e_i + \sum_{\iota = 1}^r \xi_\iota e_\iota + \sum_{j = 1}^{p} \beta_j e_{r+j},
\qquad
\xi_i \in \ZZ,
\quad
\xi_1+ \ldots + \xi_r = 0.
\]
At least one of them satisfies $\tilde u_i (\xi) \in \sigma^\vee$
and thus evaluates non-negatively on each of the generators
$v_1,\ldots, v_r \in \sigma$. Concretely that means
\[
(\xi_i+\alpha)l_i + \beta_i d_{i} \ge 0,
\quad
\xi_\iota l_\iota + \beta_\iota d_{\iota}  \ge  0,
\ \iota = 1, \ldots, p, \ \iota \ne i,
\quad
\xi_\iota \ge 0,
\ \iota = p+1, \ldots, r.
\]
The defining property of $\tilde u$ allows us to express
$\alpha$ in terms of the $\beta_\iota d_\iota/l_\iota$ and
$1/\ell$.
Together with $\xi_1+ \ldots + \xi_r = 0$, this brings us
to 
\[
- \alpha - \beta_i \frac{d_i}{l_i} 
\ \le \
\xi_i
\ \le \
- \alpha - \beta_i \frac{d_i}{l_i} + \frac{1}{\ell}.
\]
Now writing $\beta_id_i/l_i = b_i + \delta_i/l_i$ with
integers $b_i$ and $0 \le \delta_i < l_i$ and
adding $\alpha+b_i$ to each side of the inequalities,
we obtain 
\[
-\frac{\delta_i}{l_i}
\ \le \
\xi_i + \alpha + b_i
\ \le \
-\frac{\delta_i}{l_i}
+
\frac{1}{\ell}.
\]
Due to $\xi_i + \alpha + b_i \in \ZZ$, we can
replace the l.h.s.~bound with zero.
Thus, $\ell_i\delta_i \le 1$, hence
$\ell_i=\delta_i=1$ or $\delta_i=0$.
If $\ell_i=\delta_i=1$ for some $i$,
we are done. Otherwise
\[
\delta_1 \ = \ \ldots \ = \ \delta_p \ = \ 0.
\]
This implies $l_i \mid \beta_i$ for $i = 1, \ldots, p$,
say $\beta_i = \beta_i'l_i$. Then
$\ell(\alpha + \beta_1'd_1+\ldots+ \beta_p'd_p)=1$
by the defining property of $\tilde u$.
Thus, $\ell=1$ and we are done es well.

\medskip

\noindent
\emph{Step 3.}
We treat the general case.
First we perform a series of reductions
to the setting of Step~1.
Fix $1 \le j \le s$.
Then, for $i = 1,\ldots, r$, consider
\[
c_{ij} := \gcd(l_i,d_{ij}),
\quad
l_{ij}'  := l_i/c_{ij},
\quad
d_{ij}'  :=  d_{ij}/c_{ij},
\quad
v_{ij}'  :=  l_{ij}'e_i + d_{ij}'e_{r+1} \in \ZZ^{r+1}.
\]
Then the projection $\ZZ^{r+s} \to \ZZ^{r+1}$,
sending onto the coordinates
$e_1,\ldots,e_r,e_{r+j}$
maps $v_i$ to $c_{ij}v_{ij}'$ and
hence $\sigma \subseteq \QQ^{r+s}$ onto
\[
\sigma'
\ := \
\cone(v_{1j}', \ldots, v_{rj}')
\ \subseteq \
\QQ^{r+1}.
\]
Except $l_{1j}' \ge \ldots \ge l_{rj}'$, the cone
$\sigma'$ and $\varphi' \colon \TT^2 \to \TT^{r+1}$,
defined accordingly, satisfy all assumptions of
the Lemma.
Indeed, we have a commutative diagram
\[
\xymatrix{
X
\ar@{}[r]|\subseteq
\ar[d]_{\pi}
&
Z
\ar[d]^{\pi}
\\
X'
\ar@{}[r]|\subseteq
&
Z',
}
\]
where $\pi \colon Z \to Z'$ is the toric morphism
given by the projection $\ZZ^{r+s} \to \ZZ^{r+1}$.
Note that $\pi$ is the good quotient for the
action of the subtorus
\[
\TT_j
\ := \
\{t \in \TT^{r+s}; \ t_i = 1, \ i=r+1, \ldots, r+s, \ i \ne j\}
\ \subseteq \
\TT^{r+s}.
\]
The definition of $\varphi$ and $\varphi'$ ensure
that $X$ is invariant under $\TT_j$ and $X' = \pi(X)$.
Consequently, $X' = X \quot \TT_j$ inherits normality
from $X$.

Now Step~1 tells us that for very
$j = 1, \ldots, r$, at most one
of $l_{1j}', \ldots, l_{rj}'$ 
differs from one.
Let $(i_1,j_1), \ldots, (i_p,j_p)$
where $1 \le p \le s$, pick precisely
the $l_{ij}'>1$.
Then 
\[
l_{i_k} \ = \ l_{i_kj_k}'c_{i_kj_k} \ > \ c_{i_kj_k},
\qquad
l_i \ = \ c_{ij_k}, \quad \ i \ne i_k.
\]
Whenever $l_i = c_{ij}$ holds, we have $l_i \mid d_{ij}$,
hence $d_{ij}=0$.
Applying a suitable unimodular $s \times s$ matrix
to the last $s$ coordinates yields
\[
v_{i_k} \ = \ l_{i_k} e_{i_k} + d_{j_k} e_{r + j_k},
\qquad
v_i \ = \ e_i, \quad  i = 1, \ldots, r, \ i \ne i_k
\]
with $d_{j_k} := \gcd(d_{1j_k}, \ldots, d_{sj_k})$
and hence $\gcd(l_{i_k},d_{j_k}) = c_{i_kj_k} = 1$.
Suitably renumbering the $v_i$, we arrive at the
setting of Step~2, which gives the assertion.
\end{proof}


\section{Constructing toric degenerations}
\label{sec:constr-tordeg}

We construct toric degenerations of rational varieties
with a torus action of complexity
one in terms of the defining data $(A,P,\Sigma)$.
The subsequent Construction~\ref{constr:deformdata}
and Proposition~\ref{prop:fibers} reproduce in
particular~\cite[Prop.~4.1,~4.2 and Rem.~4.3,~4.4]{IlSu},
treating the projective case via the approach of
polyhedral divisors from~\cites{AlHa,AlHaSu}.
Note that our setting is more general and covers
for instance the semiprojective and thus in particular
the affine case, as needed in Section~\ref{sec:polystability}.

\begin{construction}
\label{constr:deformdata}
Consider defining data $(A,P,\Sigma)$. Given
integers $0 \le \kappa \le r$ and $\ell \ge 1$,
we define a vector in $\ZZ^{r+s}$ by
$$
\nu_{\kappa}
\ := \
\begin{cases}
-\ell e_0 - \ldots - \ell e_r,
&
\kappa = 0,
\\
\ell e_\kappa,
&
\kappa > 0
\end{cases}
$$
and obtain new defining data
$(A,P_\kappa,\Sigma_\kappa)$ with an
$(r+s+1) \times (n+1+m)$ matrix~$P_\kappa$ and a
fan $\Sigma_\kappa$ in $\ZZ^{r+s+1}$ via the
following procedure.
\begin{enumerate}
\item
The matrix $P_\kappa$ arises from $P$ by first
appending a zero row at the bottom and then
inserting $(\nu_{\kappa}, 1)$ as a new
column at the place $ij$ with $i = \kappa$ and
$j= n_\kappa +1$.
\item
The fan $\Sigma_\kappa$ in $\ZZ^{r+s+1}$ has the
maximal cones
$(\sigma \times 0) + \varrho_{\kappa n_\kappa +1}$,
where $\sigma$ runs through the maximal cones of
$\Sigma$ and $\varrho_{\kappa n_\kappa +1}$
denotes the ray through the new column.
\end{enumerate}
We will denote by $Z$ the toric variety defined
by $\Sigma$ and by $X \subseteq Z$ the
$\TT^{s}$-variety arising from $(A,P,\Sigma)$.
Similarly, $Z_\kappa$ is the toric variety
defined by $\Sigma_\kappa$ and
$\mathcal{X}_\kappa \subseteq Z_\kappa$ the
$\TT^{s+1}$-variety arising from
$(A,P_\kappa,\Sigma_\kappa)$.
\end{construction}

We take a look at the geometry of this
construction. First, let us see in detail how the
involved ambient toric varieties $Z_\kappa$ and
$Z$ interact.

\begin{remark}
\label{rem:ambientfamily}
In the setting of Construction~\ref{constr:deformdata},
let $F_\kappa \colon \ZZ^{r+s+1} \to \ZZ^{r+s+1}$ be
the linear isomorphism keeping $e_i$ fixed for
$i = 1,\ldots, r+s$ and sending $e_{r+s+1}$ to
the vector $(-\nu_\kappa, 1)$. Then we
have a commutative diagram, where both downwards
arrows represent the projection onto the
$(r+s+1)$-th coordinate:
\[
\xymatrix{
{\ZZ^{r+s+1}}
\ar[rr]^{F_\kappa}_{\cong}
\ar[dr]
&&
{\ZZ^{r+s+1}}
\ar[dl]
\\
&
{\ZZ}
&
}
\]
The map $F_\kappa$ is an isomorphism of fans from
$\Sigma_\kappa$ in $\ZZ^{r+s+1}$  to the fan
product of~$\Sigma$ in $\ZZ^{r+s}$ and the fan of
faces of $\QQ_{\ge 0}$ in $\ZZ$. Accordingly, we
have a commutative diagram of the associated
toric morphisms, involving the ambient toric
varieties $Z_\kappa$ of $\mathcal{X}_\kappa$ and
$Z$ of $X$:
\[
\xymatrix{
Z_\kappa
\ar[rr]^{\varphi_\kappa}_{\cong}
\ar[dr]_{\Psi_\kappa}
&&
Z \times \CC
\ar[dl]^{\pr_\CC}
\\
&
{\CC}
&
}
\]
Observe that $\Psi_\kappa$ is given in Cox
coordinates by
$[z_{ij},z_k] \mapsto z_{\kappa n_\kappa+1}$.
Moreover, in terms of the acting tori
$\TT^{r+s+1}$ of $Z_\kappa$ and $\CC^*$ of $\CC$,
the map $\Psi_\kappa$ sends an element
$t=(t_1, \ldots, t_{r+s+1})$ to its last
coordinate $t_{r+s+1}$.
In particular, for all
points $z \in Z_\kappa$ and all
$t \in \TT^{r+s+1}$, we have
\[
\Psi_\kappa(t \cdot z)
\ = \
t_{r+s+1} \Psi_\kappa(z).
\]
Finally, the fiber $\Psi_\kappa^{-1}(0)$ equals
the toric prime divisor of $Z_\kappa$ defined
by the ray through $v_{\kappa n_\kappa+1} = (\nu_\kappa,1)$
and thus, being a toric orbit closure, it comes with
the structure of a toric variety. We will
identify the toric variety $\Psi_\kappa^{-1}(0)$
with $Z$ via the toric morphism given by
\[
\ZZ^{r+s+1} / \ZZ v_{\kappa n_\kappa+1}
\ \to \
\ZZ^{r+s},
\qquad
v + \ZZ v_{\kappa n_\kappa+1}
\ \mapsto \
\pr_{\ZZ^{r+s}} \circ F_\kappa(v).
\]
\end{remark}

Now we examine the family
$\mathcal{X}_\kappa \to \CC$. The first of the
subsequent two remarks relates the Cox ring of
$\mathcal{X}_\kappa$ to that of $X$. In the
second one, we take a look at the fibers and at
the equivariance properties of the family.

\begin{remark}
\label{rem:CRoffamily}
Consider the $\TT^s$-variety $X$ arising from
$(A,P,\Sigma)$ as in Construction~\ref{constr:RAPdown}.
Recall that we have $\Cl(X) = K$ for the divisor
class group and that the Cox ring is given by 
\[
\mathcal{R}(X)
\ = \
\CC[T_{ij},S_k] / \bangle{g_2, \ldots, g_r},
\qquad
g_\iota
 = 
\det
\left[
\begin{array}{ccc}
T_0^{l_0} & T_1^{l_1} & T_\iota^{l_\iota}
\\
a_0 & a_1 & a_\iota
\end{array}
\right],
\]
where the $K$-degrees of the $T_{ij}$ and $S_k$
are the classes of the basis vectors $e_{ij}$
and~$e_k$ in $K = \ZZ^{n+m}/ \im(P^*)$, respectively.
Now consider the $\TT^{s+1}$-variety
$\mathcal{X}_\kappa$ arising from the data
$(A,P_\kappa,\Sigma_\kappa)$ as in
Construction~\ref{constr:deformdata}.
Then we have 
\[
\Cl(\mathcal{X}_\kappa)
\ = \
\ZZ^{n+1+m}/ \im(P_\kappa^*)
\ = \
K,
\qquad
\mathcal{R}(\mathcal{X}_\kappa)
\ = \
\CC[T_{ij},S_k] / \bangle{g_{\kappa, 2}, \ldots, g_{\kappa, r}},
\]
where the new variable $T_{\kappa n_\kappa +1}$
is of $K$-degree zero and all other variables
$T_{ij}$ and~$S_k$ have the same $K$-degree in
$\mathcal{R}(\mathcal{X}_\kappa)$ as they have in
$\mathcal{R}(X)$. Moreover, the defining
relations $g_{\kappa, \iota}$ arise from the
$g_\iota$ by replacing 
$T_\kappa^{l_\kappa}$ with
$T_\kappa^{l_\kappa}T_{\kappa n_\kappa +1}^\ell$.
\end{remark}

\begin{remark}
\label{rem:firstpropsfam}
Consider $X \subseteq Z$ and
$\mathcal{X}_\kappa \subseteq Z_\kappa$
as in Construction~\ref{constr:deformdata}.
Restricting $\Psi_\kappa \colon Z_\kappa \to \CC$
from Remark~\ref{rem:ambientfamily} gives a
morphism
$\psi_\kappa \colon \mathcal{X}_\kappa \to \CC$.
The vanishing ideal of the fiber over
$\zeta \in \CC$ in Cox coordinates is given as
\[
I(\psi_\kappa^{-1}(\zeta))
\ = \
\bangle{g_{\kappa, 2}, \ldots, g_{\kappa, r}}
+
\bangle{T_{\kappa n_\kappa+1} - \zeta}.
\]
Moreover, the morphism
$\psi_\kappa \colon \mathcal{X}_\kappa \to \CC$
is compatible with the $\TT^{s+1}$-action on
$\mathcal{X}_\kappa$ and the multiplication on
$\CC$ in the sense that for every point
$z \in \mathcal{X}_\kappa$ and every element
$t = (t_1,\ldots, t_{s+1}) \in \TT^{s+1}$ we
have 
\[
\psi_\kappa(t \cdot z)
\ = \
t_{s+1} \psi_\kappa(z).
\]
In particular, for any $v \in \ZZ^{s+1}$ of the
form $v = (v_1,\ldots,v_s,1)$, the corresponding
one parameter subgroup
$\lambda_v \colon \CC^* \to \TT^{s+1}$ gives a
$\CC^*$-action on $\mathcal{X}_\kappa$ such that
for all $z \in \mathcal{X}_\kappa$ and
$t \in \CC^*$ we have
\[
\psi_\kappa (\lambda_v(t) \cdot x) 
\ = \
t \psi_\kappa(z).
\]
With every one parameter subgroup
$\lambda_v \colon \CC^* \to \TT^{s+1}$ with
$v=(v_1,\ldots,v_s,1)$ as above,
we associate an automorphism
$\alpha_v \colon \TT^{s+1} \to \TT^{s+1}$
and a twist of the action of $\TT^{s+1}$ on
$Z_\kappa$ and $\mathcal{X}_\kappa$
by setting
$$
\alpha_v(t)
\ := \ 
(t_1t_{s+1}^{v_1},\ldots,t_st_{s+1}^{v_s},t_{s+1}),
\qquad
t \ast z
\ := \
\alpha(t) \cdot z.
$$
\end{remark}

We gather the relevant properties of the
family $\psi_\kappa \colon \mathcal{X}_\kappa \to \CC$,
its fibers and, in particular, the special
fiber $\mathcal{X}_{\kappa,0}$. It turns out that this a
possibly non-normal toric variety. We determine
the associated fan and
characterize, when it is normal.

\begin{construction}
\label{constr:fankappa}
Consider defining data $(A,P,\Sigma)$.
The leaves of the associated tropical
variety are the cones
\[
\tau_{i}
\ = \
\cone(e_i) + \lin(e_{r+1}, \ldots, e_{r+s}),
\qquad\qquad
i \ = \ 0, \ldots, r,
\]
where $e_1, \ldots, e_{r+s} \in \QQ^{r+s}$ are
the canonical basis vectors and
$e_{0}=-e_{1}-\dots-e_{r}$. Then, for
$\kappa=0,\dots, r$, we obtain lattice fans
$(\Delta_\kappa,N_\kappa)$ by setting
\[
V_{\kappa} \ := \ \lin(\tau_\kappa),
\qquad
N_\kappa \ := \ \ZZ^{r+s} \cap V_\kappa,
\qquad
\Delta_\kappa \ := \ \{\sigma \cap V_\kappa; \ \sigma \in \Sigma\}.
\]
We define the \emph{antitropical coordinates} of a vector
$v \in N_\kappa$ to be $\eta_\kappa^{-1}(v) \in \ZZ^{s+1}$,
where 
\[
\eta_\kappa \colon \ZZ^{s+1} \ \to \ N_\kappa,
\qquad
e_i \ \mapsto \
\begin{cases}
e_{r+i},  & i = 1 ,\ldots, s,
\\
-e_{\kappa}, & i = s+1.
\end{cases}
\]
\item
We denote by $\Delta_\kappa^{\mathrm{at}}$ the fan in $\ZZ^{s+1}$
representing $(\Delta_\kappa,N_\kappa)$ in antitropical coordinates,
that means
\[
\Delta_\kappa^{\mathrm{at}}
\ = \
\{\eta_\kappa^{-1}(\sigma); \ \sigma \in \Sigma\}.
\]
Moreover, the \emph{antitropical half space} in $\QQ^{s+1}$
is the set of points having non-negative last
antitropical coordinate:
\[
\mathcal{H}_\kappa
\ := \
\{v \in \QQ^{s+1}; \ v_{s+1} \ge 0\}.
\]
\end{construction}

\begin{proposition}
\label{prop:fibers}
Consider $X = X(A,P,\Sigma) \subseteq Z$
and $\psi_\kappa \colon \mathcal{X}_\kappa \to \CC$
as provided by Construction~\ref{constr:deformdata}
and Remark~\ref{rem:firstpropsfam}. For $\zeta \in \CC$ write
$\mathcal{X}_{\kappa,\zeta} := \psi_\kappa^{-1}(\zeta)$.
\begin{enumerate}
\item
The variety $\mathcal{X}_\kappa$ is irreducible
and normal and 
$\psi_\kappa \colon \mathcal{X}_\kappa \to \CC$
is a flat family.
Moreover, $\psi_\kappa$ is proper (projective)
if $Z$ is complete (projective).
\item
For $\zeta \ne 0$, we have
$\mathcal{X}_{\kappa,\zeta} \cong X$.
The fiber $\mathcal{X}_{\kappa,0}$ is the closure
of a $\TT^{s+1}$-orbit in $\mathcal{X}_\kappa$
and hence an irreducible toric variety.
\item
As a toric variety, the fiber $\mathcal{X}_{\kappa,0}$
is isomorphic to the closure of the subtorus
$\TT_\kappa \subseteq \TT^{r+s} \subseteq Z$ given by
\[
\TT_\kappa
\ = \
\begin{cases}
t_1=t_{i}, \ i = 2 , \ldots, r & \kappa = 0,
\\[.5ex]
t_i= 1, \ i = 1 , \ldots, r, \, i \ne \kappa,
& \kappa \ne 0.
\end{cases}
\]  
Moreover, $\mathcal{X}_{\kappa,0}$ has $(\Delta_\kappa,N_\kappa)$
as its convergency fan and the toric variety 
associated with $(\Delta_\kappa,N_\kappa)$ as its
normalization.
\item
For $J = (j_i; \, i = 0, \ldots, r, \, i \ne \kappa)$,
set $\sigma_J := \cone(v_{ij_i}; \, i = 0, \ldots, r, \, i \ne \kappa)$.
Then each $\sigma_J \in \Sigma$ defines a prime divisor
$$
D(\sigma_J)
\ := \
\bigcap_{i \ne \kappa} D_{ij_i}
\subseteq
\mathcal{X}_{\kappa,0}.
$$
Moreover, the toric boundary 
$\mathcal{X}_{\kappa,0} \setminus \TT^{r+s}$
decomposes into a union of prime divisors
as follows:
$$
\qquad\qquad
\mathcal{X}_{\kappa,0} \setminus \TT^{r+s}
\ = \ 
\bigcup_{\sigma_J \in \Sigma} D(\sigma_J)
\ \cup \
\bigcup_{j = 1}^{n_\kappa} D_{\kappa j} \cap \mathcal{X}_{\kappa,0}
\ \cup \
\bigcup_{k = 1}^{m} D_k \cap \mathcal{X}_{\kappa,0}.
$$
\item
The fiber $\mathcal{X}_{\kappa,0}$ is normal if
and only if for every cone
$\sigma(J) \in \Sigma$,
we have $l_{j_i} > 1$ for at most one $0 \le i \le r$
distinct from $\kappa$. 
\item
The fiber $\mathcal{X}_{\kappa,0}$ is affine
(semiprojective, projective, complete) if $Z$
is affine (semiprojective, projective,
complete).
\end{enumerate}
\end{proposition}

\begin{lemma}
\label{lem:binomial}
Consider $\Psi_\kappa \colon Z_\kappa \to \CC$
and its restriction
$\psi_\kappa \colon \mathcal{X}_\kappa \to \CC$.
Then we have 
$\mathcal{X}_{\kappa,0} = \psi_\kappa^{-1}(0)
\subseteq \Psi_\kappa^{-1}(0) = Z$. Moreover,
with suitable $b_i \in \CC^*$ the vanishing
ideal $\mathcal{I}_{\kappa,0}$ of 
$\mathcal{X}_{\kappa,0}$ in Cox coordinates of
$Z$ is the $K$-prime binomial ideal
\[
\begin{array}{lclcl}
\mathcal{I}_{\kappa,0}
& = & 
\bangle{T_1^{l_1} + b_iT_{i}^{l_{i}}; \ i = 2, \ldots, r},
& &
\text{if } \kappa = 0,
\\[2ex]
\mathcal{I}_{\kappa,0}
& = & 
\bangle{T_0^{l_0} + b_iT_{i}^{l_{i}}; \ i = 1, \ldots, r, \ i \ne \kappa},
& &
\text{if } \kappa \ne 0.
\end{array}
\]
\end{lemma}

\begin{proof}
The first statement is clear by construction.
Moreover using the specific nature of the
defining trinomial relations $g_{\kappa, \iota}$
we obtain the shape of $\mathcal{I}_{\kappa,0}$,
see also the proof of~\cite[Prop.~10.7]{HaHe}.
Finally, $K$-primeness is ensured by~\cite[Prop.~10.7]{HaHe}.
\end{proof}

\begin{proof}[Proof of Proposition~\ref{prop:fibers}]
Assertions~(i) and~(vi) are clear by construction.
We prove~(ii). From Remarks~\ref{rem:CRoffamily}
and~\ref{rem:firstpropsfam} we infer
$\mathcal{X}_{\kappa,1} \cong X$. Using
equivariance of $\psi_\kappa$, we see
$\mathcal{X}_{\kappa,\zeta} \cong X$ for all
$\zeta \in \CC^*$. The fiber
$\mathcal{X}_{\kappa,0}$ is the intersection of
$\mathcal{X}_{\kappa}$ with the toric divisor of
$Z_\kappa$ given by the ray 
$\varrho_{\kappa n_{\kappa}+1} \in \Sigma_\kappa$.
In particular,
$\mathcal{X}_{\kappa,0}$ is the closure of a
$\TT^{s+1}$-orbit in $\mathcal{X}_{\kappa}$.

We prove~(iii). Look at the embedding
$\mathcal{X}_{\kappa,0} = \psi_\kappa^{-1}(0) \subseteq \Psi_\kappa^{-1}(0) = Z$.
The coordinate functions
$\chi_1, \ldots, \chi_r$ and
$\eta_1, \ldots, \eta_s$  on the acting torus
$\TT^{r+s} \subseteq Z$ satisfy
\[
p^*(\chi_i) \ = \ \frac{T_i^{l_i}}{T_0^{l_0}},
\quad
i = 1,\ldots, r,
\qquad
p^*(\eta_j) \ = \ S_j,
\quad
j = 1, \ldots, s.
\]
Using Lemma~\ref{lem:binomial}, we conclude that
$\mathcal{X}_{\kappa,0} \cap \TT^{r+s}$
is described by binomials of the form
$\chi_1 + b_i\chi_i$ if $\kappa = 0$ and
$1 + b_i\chi_i$ for $\kappa \ne 0$. Thus,
applying a suitable element
$t(\kappa) \in \TT^{r+s}$, we obtain
\begin{eqnarray*}
t(\kappa) \cdot (\mathcal{X}_{\kappa,0} \cap \TT^{r+s})
& = &
V(\chi_1-\chi_{i}; \ i = 2 , \ldots, r),
\\[.5ex]
t(\kappa) \cdot (\mathcal{X}_{\kappa,0} \cap \TT^{r+s})
& = &
V(\chi_i - 1; \ i = 1 , \ldots, r, \ i \ne \kappa),
\end{eqnarray*}
according to $\kappa = 0$ and $\kappa \ne 0$. For
every $\kappa = 0, \ldots, r$, this gives a
subtorus $\TT_\kappa \subseteq \TT^{r+s}$
corresponding to the sublattice
$N_\kappa \subseteq \ZZ^{r+s}$.
By construction, 
\[
\mathcal{X}_{\kappa,0}
\ = \
\overline{t(\kappa)^{-1}\cdot \TT_{\kappa}}
\ = \
t(\kappa)^{-1} \cdot \overline{\TT_\kappa}
\ \subseteq \
Z.
\]

We turn to~(iv). First we look at
$\bar{\mathcal{X}}_{\kappa,0} := V(\mathcal{I}_{\kappa,0}) \subseteq \CC^{n+m}$.
Given a sequence $J = (j_i; \ i = 0, \ldots, r, \ i \ne \kappa)$
with $1 \le j_i \le n_i$, set 
$$
D(J)
\ := \
V(T_{ij_i};  \ i = 0, \ldots, r, \ i \ne \kappa)
\ \subseteq \
\bar{\mathcal{X}}_{\kappa,0}.
$$
Each such $D(J)$ is isomorphic to $\CC^{n+m-r}$ and
$\bar{\mathcal{X}}_{\kappa,0} \cap V(T_{ij})$ is the
union over all those $D(J)$ with $j_i = j$.
The latter implies
$$
\bar{\mathcal{X}}_{\kappa,0} \setminus \TT^{n+m}
\ = \
\bigcup D(J)
\, \cup \,
\bigcup_{j=1}^{n_\kappa} V(T_{\kappa j}) \cap \bar{\mathcal{X}}_{\kappa,0}
\, \cup \,
\bigcup_{k=1}^m V(T_k) \cap \bar{\mathcal{X}}_{\kappa,0} .
$$
Now consider
$\hat{\mathcal{X}}_{\kappa,0} :=  \bar{\mathcal{X}}_{\kappa,0} \cap \hat Z$
and the quotient map
$p \colon \hat{\mathcal{X}}_{\kappa,0} \to \mathcal{X}_{\kappa,0}$.
Then $p$ respects the above decomposition which gives~(iv).

For~(v) note that
the Jacobian of the binomials from Lemma~\ref{lem:binomial}
is of rank at most $r-2$ along $D(J)$ if and only if $J$ admits
$i_1 \ne i_2$ with $l_{i_1j_{i_1}}, l_{i_2j_{i_2}} > 1$.
As a consequence, 
$\hat{\mathcal{X}}_{\kappa,0} = \bar{\mathcal{X}}_{\kappa,0} \cap \hat Z$
is normal if it intersects no such $D(J)$.
In this case, also
$\mathcal{X}_{\kappa,0} = p(\hat{\mathcal{X}}_{\kappa,0})$
is normal, proving the ``if'' part.
The ``only if'' part follows from
Proposition~\ref{lem:loch}.
\end{proof}

\begin{corollary}
If $X = X(A,P)$ admits a normal toric degeneration
$\mathcal{X}_{\kappa,0}$, then $X$ is quasismooth.
\end{corollary}

\begin{proof}
According to~\cite[Prop.~7.14]{HaHiWr},
the characterizing property of normality from 
Proposition~\ref{prop:fibers}~(iii)
implies quasismoothness.
\end{proof}

\begin{example}
\label{example:lbsp5}
Consider again the surface $X = X(A,P,\Sigma)$
from Examples~\ref{example:lbsp1}
and~\ref{example:lbsp2}.
The matrices $P_\kappa$, where $\kappa = 0,1,2$,
from Construction~\ref{constr:deformdata}
are given as

{\tiny
\setlength{\arraycolsep}{2pt}
$$
P_0
 = 
\left[
\begin{array}{rrrrrr}
-2 & -1 & - \ell & 1 & 1 & 0 
\\
-2 & -1 & - \ell & 0 & 0 & 2
\\
3  & -1 &  0 & 0 & -1 & 1
\\
0 & 0 & 1 & 0 & 0 & 0 
\end{array}
\right],
\qquad
P_1
\ = \
\left[
\begin{array}{rrrrrr}
-2 & -1 & 1 & 1 &  \ell & 0
\\
-2 & -1 & 0 & 0 & 0 & 2
\\
3  & -1 & 0 & -1 & 0 & 1
\\
0 & 0 & 0 & 0 & 1 & 0 
\end{array}
\right],
\qquad
P_2
 = 
\left[
\begin{array}{rrrrrr}
-2 & -1 & 1 & 1 & 0 & 0
\\
-2 & -1 & 0 & 0 & 2 & \ell
\\
3  & -1 & 0 & -1 & 1 & 0
\\
0 & 0 & 0 & 0 & 0 & 1
\end{array}
\right].
$$
}

\noindent
In order to obtain proper flat families $\mathcal{X}_{\kappa}$,
we replace $Z$ with its projective closures from
Example~\ref{example:lbsp3}. This gives us the
embeddings
$$
X \ \subseteq \ Z_{w_{02} + w_{21}},
\qquad
X \ \subseteq \ Z_{w_{21}},
\qquad
X \ \subseteq \ Z_{w_{21} + w_{01}}.
$$
In either case, Lemma~\ref{lem:binomial}
tells us that the fibers
$\mathcal{X}_{\kappa,0} = \psi_\kappa^{-1}(0) \subseteq \Psi_\kappa^{-1}(0) = Z$
are given in Cox coordinates by binomials
as follows:
$$
\mathcal{X}_{0,0}  =  V(T_{11}T_{12} + T_{12}^2),
\quad
\mathcal{X}_{1,0} = V(T_{01}^2T_{02} + T_{12}^2),
\quad
\mathcal{X}_{2,0} = V(T_{01}^2T_{02} + T_{11}T_{12}).
$$
Moreover, Proposition~\ref{prop:fibers} says that $\mathcal{X}_{0,0}$
and $\mathcal{X}_{2,0}$ are normal, whereas
$\mathcal{X}_{1,0}$ is not.
Here are the generator matrices of the fans $\Delta^{\mathrm{at}}_\kappa$
of the special fibers:

\bigskip

\begin{center}
\begin{tabular}{c|c|c|c}
&
$\kappa = 0$
&
$\kappa = 1$
&
$\kappa = 2$
\\[2pt] \hline &&& \\[-8pt]
$w_{02} + w_{21}$
&
{\tiny
\setlength{\arraycolsep}{2pt}
$   
\left[\begin{array}{rrrr}
-1 & -1 & 1 & 3 
\\
-1 &  2 & 2 & -2   
\end{array}\right]$
}
&
{\tiny
\setlength{\arraycolsep}{2pt}
$\left[\begin{array}{rrrr}
-1 & -1 & 0 & 2
\\
-1 &  2 & -1 & 1   
\end{array}\right]$
}
&
{\tiny
\setlength{\arraycolsep}{2pt}
$\left[\begin{array}{rrrr}
-2 & 1 & 1 & 3 
\\
-1 &  -2 & 2 & 2   
\end{array}\right]$
}
\\[8pt] \hline &&& \\[-8pt]
$w_{21}$
&
{\tiny
\setlength{\arraycolsep}{2pt}
$   
\left[\begin{array}{rrrr}
-1 & -1 & 1 & 3 
\\
-1 &  2 & 2 & -2   
\end{array}\right]$
}
&
{\tiny
\setlength{\arraycolsep}{2pt}
$\left[\begin{array}{rrrr}
-1 & -1 & 0 & 2 
\\
-1 &  2 & -1 & 1   
\end{array}\right]$
}
&
{\tiny
\setlength{\arraycolsep}{2pt}
$\left[\begin{array}{rrr}
-2 & 1 & 3 
\\
1 & -2 & 2   
\end{array}\right]$
}
\\[8pt] \hline &&& \\[-8pt]
$w_{21}+w_{01}$
&
{\tiny
\setlength{\arraycolsep}{2pt}
$   
\left[\begin{array}{rrrr}
-1 & -1 & 1 & 3 
\\
-1 &  2 & 2 & -2   
\end{array}\right]$
}
&
{\tiny
\setlength{\arraycolsep}{2pt}
$\left[\begin{array}{rrrr}
-1 & -1 & 0 & 2
\\
-1 &  2 & -1 & 1   
\end{array}\right]$
}
&
{\tiny
\setlength{\arraycolsep}{2pt}
$\left[\begin{array}{rrrr}
-2 & -1 & 1 & 3 
\\
1 &  1 & -2 & 2   
\end{array}\right]$
}
\end{tabular}
\end{center}

\bigskip

\end{example}

\begin{remark}
\label{rem:antitropical-oneparameter}
Consider
$\psi_\kappa \colon \mathcal{X}_{\kappa} \to \CC$.
Let $\lambda \colon \CC^* \to \TT^{s+1}$ be a one
parameter subgroup acting on $\mathcal{X}_\kappa$
such that
$\psi_\kappa(\lambda(t) \cdot z) = t \psi_\kappa(z)$.
\begin{enumerate}
\item
On $\Psi_\kappa^{-1}(0) = Z$, the one parameter
subgroup $\lambda$ is given by
$v - \nu_\kappa \in \ZZ^{r+s}$ with
a unique $v \in \{0\} \times \ZZ^s$.
\item
On
$\mathcal{X}_{\kappa,0} = \psi_\kappa^{-1}(0)$,
the one parameter subgroup $\lambda$ is given in
antitropical coordinates by
$(v,\ell) \in \mathcal{H}_\kappa$. 
\end{enumerate}
Conversely, every
$(v,\ell) \in \ZZ^{s+1}$
is the antitropical coordinate vector of a one
parameter subgroup $\lambda \colon \CC^* \to \TT^{s+1}$
acting on $\mathcal{X}_\kappa$ such that 
$\psi_\kappa(\lambda(t) \cdot z) = t \psi_\kappa(z)$.
\end{remark}


\section{Equivariant test configurations}
\label{sec:testconf}

The toric degenerations discussed in the preceding
section deliver indeed equivariant test
configurations.
We show that we obtain in fact all equivariant
test configurations for a semiprojective variety
in this way and we relate the equivariant test
configurations of a Fano variety with
torus action of complexity one to those of its
anticanonical cone.
We deal with \emph{polarized $T$-varieties}
$(X,L)$, that means that $X$ is a normal variety
endowed with an effective torus action
$T \times X \to X$ and an ample 
divisor $L$ coming with a $T$-linearization
on $\mathcal{O}(L)$.

\begin{remark}
\label{rem:fanopolar}
For an $n$-dimensional Fano $T$-variety $X$,
we have natural polarizations:
Take any $T$-invariant anticanonical 
divisor $L$ on $X$ and endow
$\mathcal{O}(L)$ with the unique
$T$-linearization such that
$\mathcal{O}(n \iota_X L)$ 
is isomorphic as a $T$-sheaf to the 
$\iota_X$-fold tensor power of the 
$n$-fold exterior power of the 
$T$-equivariant tangent sheaf of~$X$.
\end{remark}

\begin{definition}
\label{def:testconfig}
An \emph{equivariant test configuration}
for a polarized $T$-variety $(X,L)$ is
a triple $(\mathcal{X},\mathcal{L},\psi)$,
where
\begin{enumerate}
\item 
$(\mathcal{X},\mathcal{L})$
is a polarized $\mathcal{T}$-variety
with $\mathcal{T} = T \times \CC^*$,
\item
$\psi \colon \mathcal{X} \to \CC$
is a flat, 
$T$-invariant, $\CC^*$-equivariant
morphism,
\item
the fibre $\mathcal{X}_1 = \psi^{-1}(1)$ is $T$-equivariantly
isomorphic to~$X$,
\item
the line bundle $\mathcal{L}$ restricts to $L$ on
$\mathcal{X}_1$.   
\end{enumerate}
An equivariant test configuration
$(\mathcal{X},\mathcal{L},\psi)$ for $(X,L)$
is called \emph{special} if the
fibre $\mathcal{X}_0 = \psi^{-1}(0)$ is a normal variety.
\end{definition}

For $T$-varieties of complexity one, the 
toric degenerations $\psi_\kappa \colon \mathcal{X}_\kappa \to \CC$
of $X$ discussed in Section~\ref{sec:constr-tordeg}
provide us with examples of equivariant test configurations.

\begin{proposition}
\label{prop:testconf-APS}
Let $X \subseteq Z$ arise from $(A,P,\Sigma)$
with a semipolytopal fan~$\Sigma$ via
Construction~\ref{constr:RAPdown}
and let $L$ be an ample toric divisor
on $Z$.
Moreover, let $\mathcal{X}_\kappa \subseteq Z_\kappa$
arise from $(A,P_\kappa,\Sigma_\kappa)$
via Construction~\ref{constr:deformdata}.
\begin{enumerate}
\item
Restricing the divisor $L$ to $X$
gives a polarized $\TT^s$-variety $(X,L)$.
Moreover, $L$ naturally extends
to an ample toric divisor $L_\kappa$ on
$Z_\kappa \cong Z \times \CC$.
\item
Restricting the toric divisor
$L_\kappa$ from~(i) to $\mathcal{X}_\kappa$
yields an equivariant test configuration
$(\mathcal{X}_\kappa,\mathcal{L}_\kappa,\psi_\kappa)$
for $(X,L)$.
\item
$(\mathcal{X}_\kappa,\mathcal{L}_\kappa,\psi_\kappa)$
from~(ii) 
is special if and only if for any
$\sigma  = \cone(v_{ij_i}; \, i \ne \kappa)$ in $\Sigma$,
we have $l_{j_i} > 1$ for at most one $0 \le i \le r$
distinct from $\kappa$. 
\item
Suppose that $X \subseteq Z$ is Fano,
$(\mathcal{X}_\kappa,\mathcal{L}_\kappa,\psi_\kappa)$
is special and, with
$0 \le \mu \le r$ such that $l_{ij} = 1$
for all $i \ne \mu,\kappa$, the divisor 
$L$ is given as
$$
\qquad\qquad
L
\ = \
\sum D_{ij} + \sum D_k - \sum_{i \ne \kappa,\mu} D_{ij}
\ = \
\sum_{j = 1}^{n_\kappa} D_{\kappa j} + \sum_{j = 1}^{n_\mu} D_{\mu j}
+ \sum D_k.
$$
Then $L$ restricts to an anticanonical divisor on
$X \subseteq Z$ and $\mathcal{L}_\kappa$ restricts
to the toric boundary divisor of $\mathcal{X}_0$;
in particular, $\mathcal{X}_0$ is Fano.
\end{enumerate}
\end{proposition}

\begin{proof}
The first two statements are clear by construction.
Assertion~(iii) follows from
Proposition~\ref{prop:fibers}~(v).
Moreover, the latter shows that we can
choose $\mu$ and $L$ as in Assertion~(iv)
and Proposition~\ref{prop:fibers}~(iv)
shows that $L$ restricts as wanted.
\end{proof}

For semiprojective $T$-varieties of complexity one
with only constant invariant functions,
the above proposition yields all possible
equivariant test configurations except the
\emph{product configurations} that means the 
$\mathcal{X} \cong X \times \CC$ with $\psi$
being the projection onto $\CC$.

\begin{proposition}
\label{prop:all-test-conf}
Let $X$ be a rational semiprojective $T$-variety $X$
of complexity one with only constant invariant functions
and let $(\mathcal{X},\mathcal{L}_\kappa,\psi)$
a non-product equivariant test configuration.
Then there is an equivariant isomorphism $X \cong X(A,P,\Sigma)$
such that $\mathcal{X}_0 \cong \mathcal{X}_{\kappa,0}$
holds for some $0 \le \kappa \le r$.
\end{proposition}

\begin{proof}
Let $D_{ij} \subseteq X$ denote the $T$-invariant prime
divisors forming the critical fibers
of the quotient map $\pi_X \colon X \dasharrow \PP_1$,
that means that for each of the critical values
$\alpha_0, \ldots, \alpha_r \in \PP_1$
of $\pi_X$ we have a dense inclusion
$$
\pi_X^{-1}(\alpha_i) \ \subseteq \ D_{i1} \cup \ldots \cup D_{n_i}.
$$
Moreover, let $D_k \subseteq X$ be the invariant prime
divisors with infinite general $T$-isotropy.
Identifying $X$ with the fibre
$\mathcal{X}_1 \subseteq \mathcal{X}$ of
$\psi \colon \mathcal{X} \to \CC$
provides us with a $(T \times \CC^*)$-equivariant
isomorphism 
$$
X \times \CC^* \ \to \ \psi^{-1}(\CC^*),
\qquad
(x,s) \ \mapsto s \cdot x,
$$
where the torus $T$ acts as given on the first
factor of $X \times \CC^*$ and $\CC^*$ on the
second one via multiplication.
Via this isomorphism, we see
$\Cl(\psi^{-1}(\CC^*)) \cong \Cl(X)$
and obtain prime divisors
$$
\CC^* \cdot D_{ij}
\ \subseteq \
\psi^{-1}(\CC^*),
\qquad
\CC^* \cdot D_{ij}
\ \subseteq \
\psi^{-1}(\CC^*),
$$
the classes of which generate $\Cl(\psi^{-1}(\CC^*))$.
Let
$\mathcal{D}_{ij},\mathcal{D}_k \subseteq \mathcal{X}$ 
denote the respective closures.
By definition
$\mathcal{X}_0 \subseteq \mathcal{X}$ 
is a principal prime divisor,
equal to $\mathcal{X} \setminus \psi^{-1}(\CC^*)$.
Consequently, we obtain
$$
\Cl(\mathcal{X}) \ \cong \ \Cl(X)
$$
and the l.h.s. group is generated by classes of
the $\mathcal{D}_{ij}$ and $\mathcal{D}_k$.
Now, set $\mathcal{T} := T \times \CC^*$
and regard $\mathcal{X}$ as a $\mathcal{T}$-variety.
We show that with a suitable defining matrix~$P$
for~$X$, a defining matrix $\mathcal{P}$ for~$\mathcal{X}$
can be choosen to be one of the 
following two:
$$
\mathcal{P}
\ = \
\left[
\begin{array}{cc}
v & P 
\\                   
1 & 0 
\end{array}
\right],
\quad
v = (-\ell, \ldots, -\ell,0, \ldots, 0),
\qquad\qquad
\mathcal{P}
\ = \
\left[
\begin{array}{cc}
P & 0 
\\                   
0 & 1 
\end{array}
\right].
$$

First we treat the case that $\mathcal{X}_0$ is the
closure of a $\mathcal{T}$-orbit.
Then $\mathcal{X}_0$ maps via $\mathcal{X} \dasharrow \PP_1$ to
the image point $\alpha \in \PP_1$ of some codimension
one orbit $\mathcal{T} \cdot x$ with
$x \in \mathcal{X}_1=X$;
recall for this that $X \dasharrow \PP_1$ maps
the set of codimension one $T$-orbits onto $\PP_1$.
Choose a defining matrix $P$ for $X$ such that
$T \cdot x \subseteq D_{01}$ holds.
Then the fibre of
$\pi_{\mathcal{X}} \colon \mathcal{X} \to \PP_1$
over $\alpha \in \PP_1$ is a dense subset of 
$$
\mathcal{X}_0 \cup \mathcal{D}_{01} \cup \ldots \cup \mathcal{D}_{0n_0}.
$$
The order of the general $\mathcal{T}$-isotropy
of $\mathcal{D}_{ij}$ equals the order $l_{ij}$
of the general $T$-isotropy of $D_{ij}$.
Thus, the $l_{ij}$ together with the order $\ell$
of the general $\mathcal{T}$-isotropy
of~$\mathcal{X}_0$ fill the first $r$
rows of $\mathcal{P}$ as claimed.
The rows $r+1$ to $r+s$ reflect the fact
that the classes of the
$\mathcal{D}_{ij}$ and $\mathcal{D}_k$
have the same relations in $\Cl(\mathcal{X})$
as the $D_{ij}$ and~$D_k$ in $\Cl(X)$
and the last row of $\mathcal{P}$ reflects
the fact that $\mathcal{X}_0$ is principal.

Now consider the case that the
$\mathcal{T}$-invariant prime divisor
$\mathcal{X}_0$ is not the closure of a single
$\mathcal{T}$-orbit. Then $\mathcal{X}_0$
must be added to the $\mathcal{D}_k$, say as
$\mathcal{D}_{m+1}$. This leads to a
matrix $\mathcal{P}$ of the second shape,
where, as before, the rows reflect
the relations in $\Cl(\mathcal{X}) \cong \Cl(X)$.
In the present case, the defining
matrix~$\mathcal{P}$ necessarily leads
to a product test configuration, which was
excluded by assumption.

So far, we determined defining matrices $P$
for $X$ and $\mathcal{P}$ for $\mathcal{X}$.
To proceed recall that the class
$[L] \in \Cl(X)$ lies in
the interior of the moving cone of any
projective toric variety $Z$ having $P$ as 
generator matrix.
Thus, there is a (unique) complete fan
$\Sigma$ such that $[L] \in \Cl(Z)$
becomes ample on its associated toric
variety~$Z$.
Now, renumbering the first columns of 
$P$ and $\mathcal{P}$ accordingly and
choosing a suitable $A$, we have the
defining data $(A,P,\Sigma)$ for $X$
and $(A,P_\kappa,\Sigma_\kappa)$
for $\mathcal{X}$, where $\kappa=0$ and
$P_\kappa = \mathcal{P}$.
The necessary linearized divisors
on $X(A,P,\Sigma) \cong X$
and $\mathcal{X}_\kappa \cong \mathcal{X}$
are transferred via the isomorphisms.
\end{proof}

We will not only have to deal with test configurations
of Fano varieties but also with test configurations
of their anticanonical cones.
The following construction and proposition provide
the necessary concepts and facts concerning the
interplay of both settings.

\begin{construction}
\label{constr:moment-polytope}
Consider a Fano variety $X \subseteq Z$
arising from $(A,P,\Sigma)$ with a
polytopal fan~$\Sigma$.
Fix $(\alpha_{ij},\alpha_k)$ such that
we obtain an anticanonical class
\[
\sum \alpha_{ij}w_{ij} + \sum \alpha_k w_k
\ \in \
K
\ = \
\Cl(X)
\ = \
\Cl(Z).
\]
For instance, $\alpha_{0j} = (1+l_{0j}-rl_{0j})$
and
$\alpha_{ij} = \alpha_k = 1$ for $i \ne 0$
will work.
As in Proposition~\ref{prop:affinecone},
set $s' := s+1$ and take the $r+s'$ by
$n+m$ stack matrix
\[
P'
\ = \
\left[
\begin{array}{c}
P
\\
\alpha
\end{array}
\right],
\qquad
\alpha
\ = \
(\alpha_{ij},\alpha_k).
\]
This gives defining data $(A,P',\Sigma')$
for an anticanonical cone $X'$ over $X$,
where~$\Sigma'$ hosts the faces of the
cone $\sigma' \subseteq \RR^{r+s'}$
over the columns of $P'$.
For $\kappa =  0, \ldots, r$ set
\[
\tau'_\kappa
\ := \
(\eta'_\kappa)^{-1} (\sigma')
\ \subseteq \ 
\RR^{s'+1},
\qquad\qquad
\omega'_\kappa 
\ := \
(\tau'_\kappa)^\vee
\ \subseteq \ 
\RR^{s'+1},
\]
where $\eta'_\kappa \colon \ZZ^{s'+1} \to N'_\kappa$
are the antitropical coordinates from Construction~\ref{constr:fankappa}.
With  $\imath \colon \RR^{s+1} \to \RR^{s'+1}$,
$u \mapsto (u_1,\ldots,u_s,1,u_{s+1})$,
we obtain for each $\kappa$ a polytope
\[
\mathcal{C}_\kappa
\ := \
\imath^{-1}(\omega'_\kappa)
\ \subseteq \
\RR^{s+1}.
\]
Observe that if $\kappa$ is special, then there is a unique
interior lattice point $u_\kappa \in \mathcal{C}_\kappa$.
We define the \emph{moment polytope} as
\[
\mathcal{B}_\kappa
\ := \
\begin{cases}
\mathcal{C}_\kappa - u_\kappa, & \kappa \text{ special},
\\
\mathcal{C}_\kappa , & \kappa \text{ not special}.
\end{cases}
\]
\end{construction}

\begin{proposition}
\label{prop:conemeetsdeg}
Consider a Fano variety $X \subseteq Z$ arising
from $(A,P,\Sigma)$ with anticanonical cone
$X' \subseteq Z'$ given by $(A,P',\Sigma')$
as in Construction~\ref{constr:moment-polytope}.
Taking
\[
L \ = \ \sum \alpha_{ij} D_{ij} + \sum \alpha_k D_k,
\qquad
L' \ = \ \sum \alpha_{ij} D'_{ij} + \sum \alpha_k D'_k
\]
with the $\alpha_{ij}$ and $\alpha_k$ used
Construction in~\ref{constr:moment-polytope}
provides us with test configurations
$(\mathcal{X}_\kappa,\mathcal{L}_\kappa,\psi_\kappa)$
for $X$ and
$(\mathcal{X}'_\kappa,\mathcal{L}'_\kappa,\psi'_\kappa)$
for $X'$ and we have the following.
\begin{enumerate}
\item
For every $\kappa = 0, \ldots, r$, the families 
$\mathcal{X}'_\kappa \to \CC$ and $\mathcal{X}_\kappa \to \CC$
fit into a commutative diagram
\[
\xymatrix{
{\mathcal{X}'_\kappa}
\ar@{-->}[rr]
\ar[dr]_{\psi'_\kappa}
& &
{\mathcal{X}_\kappa}
\ar[dl]^{\psi_\kappa}
\\
& {\CC} &
}
\]
where for each $\zeta \in \CC$ the induced map
$\mathcal{X}'_{\kappa,\zeta} \dasharrow \mathcal{X}_{\kappa,\zeta}$
of fibers is the quotient by the $\CC^*$-action on
$\mathcal{X}'_\kappa \subseteq Z'_\kappa$ 
given by $e_{r+s'} \in \ZZ^{r+s'+1}$. 
\item
For every $\kappa = 0, \ldots, r$ and every $\zeta \in \CC^*$,
the map
$X' \cong \mathcal{X}'_{\kappa,\zeta} \dasharrow \mathcal{X}_{\kappa,\zeta} \cong X$
of fibers is the anticanonical cone projection.
\item
For every $\kappa = 0, \ldots, r$, the algebra
of functions of the fiber $\mathcal{X}'_{\kappa,0}$ is
the monoid algebra $\CC[(\eta'_\kappa)^*(\omega' \cap \ZZ^{r+s'})]$,
where $\omega' = (\sigma')^\vee$.
\item
If $\kappa$ is special, then the map
$\mathcal{X}'_{\kappa,0} \dasharrow \mathcal{X}_{\kappa,0}$
of fibers is an anticanonical cone over the toric
Fano variety $\mathcal{X}_{\kappa,0}$.
\item
Let $\kappa$ be special, $v'_1,\ldots,v'_k$ the
primitive generators of $\tau'_\kappa$ and
$\tilde v_i \in \ZZ^{s'+1}$ arise from $v_i' \in \ZZ^{s'+1}$
by replacing the $s'$-coordinate with $1$. Then
\[
\tilde v_i \ = \ G_\kappa \cdot v'_i,
\qquad
i = 1, \ldots, k,
\]
with a unique unimodular matrix $G_\kappa \in \GL(s'+1,\ZZ)$.
This gives us in particular the mutually dual cones
\[
\qquad
\tau_\kappa
\ := \
G_\kappa \cdot \tau'_\kappa
\ \subseteq \
\RR^{s'+1},
\qquad
\omega_\kappa
\ := \
(G_\kappa^{-1})^* \cdot \omega'_\kappa
\ \subseteq \
\RR^{s'+1},
\]
where $\tau_\kappa$ has the primitive
generators $\tilde v_1,\ldots, \tilde v_k$ .
Let $\pr \colon \ZZ^{s'+1} \to \ZZ^{s+1}$ denote
the projection erasing the $s'$-coordinate.
Then 
\[
v_i
\ := \
\pr(v_i)
\ = \
\pr(\tilde v_i),
\qquad
i = 1, \ldots, k,
\]
are the primitive generators of
the fan $\Delta^{\mathrm{at}}_\kappa$
from Construction~\ref{constr:fankappa}.
Hence, the Fano polytope
of $\mathcal{X}_{\kappa,0}$ and its dual
are given as
\[
\mathcal{A}_\kappa
\ = \ 
\conv(v_1, \ldots, v_r),
\qquad \qquad 
\mathcal{A}_\kappa^*
\ = \
\mathcal{B}_\kappa .
\]
\end{enumerate}
\end{proposition}

\begin{proof}[Proof of Construction~\ref{constr:moment-polytope}
  and Proposition~\ref{prop:conemeetsdeg}]
By the nature of Construction~\ref{constr:deformdata},
the defining matrix
$P_\kappa$ of the family $\mathcal{X}_\kappa \to \CC$
is retrieved from the defining matrix 
$P'_\kappa$ of the family $\mathcal{X}'_\kappa \to \CC$
by removing the $s'$-th row.
This procedure reflects taking the quotient
by the $\CC^*$-action given by $e_{s'} \in \ZZ^{r+s'+1}$. 
Assertion~(i) follows.
Applying Proposition~\ref{prop:affinecone},
we obtain~(ii).
Assertion~(iii) is a consequence of
Proposition~\ref{prop:fibers}~(iii)
and the definition of the antitropical coordinates.

We proceed with the statements on the special~$\kappa$.
Note that $L'$ extends the pullback of $L$ with respect
to the cone projection $Z' \dasharrow Z$.
As a consequence,~$\mathcal{L}'_\kappa$ extends
the pullback of $\mathcal{L}_\kappa$ with respect to
the quotient map $Z_\kappa' \dasharrow Z_\kappa$.
With the identifications
$(\Psi'_\kappa)^{-1}(0) = Z'$ and $\Psi_\kappa^{-1}(0) = Z$
from Construction~\ref{constr:deformdata},
we obtain a commutative diagram of fibers
\[
\xymatrix{
{\mathcal{X}'_{\kappa,0}}
\ar@{}[r]|\subseteq
\ar@{-->}[d]
&
Z'
\ar@{-->}[d]
\\
{\mathcal{X}_{\kappa,0}}
\ar@{}[r]|\subseteq
&
Z
}  
\]
Proposition~\ref{prop:testconf-APS}~(iii)
guarantees that $L'$ and $L$ restrict to
toric anticanonical divisor on
and $\mathcal{X}'_{\kappa,0}$ and
$\mathcal{X}_{\kappa,0}$,
respectively.
This implies~(iv). We show~(v).
In antitropical coordinates,
$\mathcal{X}'_{\kappa,0}$ arises from
the cone $\tau'_\kappa$, generated
by the $v_i'$ and
$\mathcal{X}_{\kappa,0}$ arises from
the fan $\Delta_\kappa^{\mathrm{at}}$ in $\ZZ^{s+1}$,
having $v_i = \pr(v'_i)$ as its primitive
ray generators.
In this setting, the cone projection
$\mathcal{X}'_{\kappa,0} \dasharrow \mathcal{X}_{\kappa,0}$
is the quotient by the $\CC^*$-action corresponding
to $e_{s'} \in \ZZ^{s'+1}$.
In particular, (iv) together with
Proposition~\ref{prop:affinecone}
tell us that the $s'$-th row
$(v'_{1,s'}, \ldots, v'_{k,s'})$ of
the generator matrix $[v'_1, \ldots, v'_k]$
of $\tau'$ lists the coefficients of a
toric anticanonical divisor of
$\mathcal{X}_{\kappa,0}$.
Morover, also the $s'$-th row $[1,\ldots,1]$
of the generator matrix
$[\tilde v_1, \ldots, \tilde v_k]$
of $\tau_\kappa$ lists the coefficients of a
a toric anticanonical divisor on
$\mathcal{X}_{\kappa,0}$.
Again by Proposition~\ref{prop:affinecone},
we see that $\tau_\kappa$ defines an
anticanonical cone over $\mathcal{X}_{\kappa,0}$.
The two involved anticanonical divisors are
linearly equivalent and thus differ by 
an integral linear combination
over the rows of $[v_1,\ldots, v_k]$,
which reflect the relations 
among the toric prime divisors in
the divisor class group of $\mathcal{X}_{\kappa,0}$.
The desired matrix $G_\kappa$ arises from the
$s'+1$ by $s'+1$ unit matrix by replacing
the zeros of its $s'$-th row with the
coefficients of the above linear combination.
The statements concerning $\mathcal{A}_\kappa$ and
$\mathcal{B}_\kappa$ are now clear. They imply
in particular that $\mathcal{C}_\kappa$ has a
unique interior lattice point.
\end{proof}

\section{Testing polystability}
\label{sec:polystability}

We recall the necessary concepts and results
from equivariant $K$-polystability and the
general criteria for the existence of
K\"ahler-Einstein metrics, K\"ahler-Ricci solitons
and Sasaki-Einstein metrics.
Moreover, based on the preceding combinatorial
study of test configurations, we can
provide explicit versions of these 
criteria for del Pezzo $\CC^*$-surfaces
$X(A,P)$ in terms of the defining matrix $P$.

We begin with K\"ahler-Einstein metrics.
Let $(X,L)$ be a polarized Fano $T$-variety
as in Remark~\ref{rem:fanopolar}.
The section spaces of the multiples
of~$L$ decompose into eigenspaces:
$$
H^0(X,\mathcal{O}(kL))
\ = \
\bigoplus_{\mathbb{X}(T)} H^0(X,\mathcal{O}(kL))_\chi.
$$
We denote by $h_k$ and $h_{k,\chi}$ the dimensions
of the involved vector spaces.
The \emph{Futaki character} associated
with $(X,L)$ is 
$$
F_X
\ := \
- \lim_{k \to \infty}
\frac{1}{k h_k}
\sum_{\chi \in \Chi(T)}
h_{k,\chi} \chi
\ \in \
\mathbb{X}_\RR(T).
$$
For a special equivariant test configuration
$(\mathcal{X},\mathcal{L},\psi)$ for $(X,L)$,
let $\lambda \in \Lambda(\mathcal{T})$ be the
one parameter subgroup corresponding to
the $\CC^*$-part of $\mathcal{T} = T \times \CC^*$.
Then
$$
\mathrm{DF} (\mathcal{X},\mathcal{L},\psi)
\ := \
F_{\mathcal{X}_0}(\lambda)
$$
is the \emph{Donaldson-Futaki invariant} of
$(\mathcal{X},\mathcal{L},\psi)$.
Now, $(X,L)$ is \emph{equivariantly $K$-polystable} if
$\mathrm{DF} (\mathcal{X},\mathcal{L},\psi) \ge 0$ for every
special equivariant test configuration, with equality only for product test configuration.

\begin{theorem}
A Fano variety  with torus action of complexity $1$ admits
a K\"ahler-Einstein metric if and only if it is equivariantly
$K$-polystable.
\end{theorem}

\begin{proof}
Proposition~\ref{prop:all-test-conf} ensures that
$X$ has only finitely many special toric degenerations
in the sense of~\cite[Cor.~1.4]{Li}.
The latter tells us that this situation, $K$-polystability
of $X$ implies existence of a K\"ahler-Einstein metric on $X$.
Conversely, if~$X$ admits a K\"ahler-Einstein metric,
then~\cite[Thm.~1.1]{Berman} guarantees $K$-polystability
of $X$. In particular, $X$ is equivariantly $K$-polystable.
\end{proof}

\begin{theorem}
\label{thm:ke-crit}
A del Pezzo $\CC^*$-surface $X=X(A,P,\Sigma)$ admits a
K\"ahler-Einstein metric if and only if the coordinates of
the barycenter $b_\kappa \in \mathcal{B}_\kappa \subseteq \QQ^2$
satisfy $b_{\kappa,1} = 0$ for all $0 \le \kappa \le r$
and $b_{\kappa,2} > 0$ for all special $0 \le \kappa \le r$.
\end{theorem}

\begin{proof}
We first treat the product test configurations.
Here, $\operatorname{DF} (\mathcal{X},\mathcal{L},\psi)=F_X(\lambda)$.
Hence, the K-polystability condition means that the Futaki
character $F_X$ has to vanish.
In \cite{Picturebook} a formula for the Futaki character $F_X$ in
terms of so-called divisorial polytopes was given.
In our setting this is a piecewise affine and concave
function  $\Psi \colon \RR \to \operatorname{Div}_\RR(\PP^1)$
into the group of divisors on $\PP^1$ supported
on a compact interval.
Such a divisorial polytope is a combinatorial invariant associated
with every polarized $\CC^*$-surface.
Then by \cite[Thm~3.15]{Picturebook}  we have the following
formula for the Futaki character
\[
F_X(\lambda) \ = \ \int \lambda u \cdot \deg  \Psi(u) du.
\]
On the other hand, by \cite[Cor.~4.6.]{IlSu} for every test
configuration the normalization of the special fibre is a
polarized toric variety corresponding to a polytope of the form
\[
\Delta_y
\  = \
\left\{u \in \RR^2; \
{\scriptstyle -\sum_{z\in \PP^1\setminus\{ y\}}} \Psi_z(u_1) \, \leq \, u_2 \, \leq \, \Psi_y(u_1)\right\},
\]
where $y \in \PP^1$ and $\Psi_y(u)$ is the coefficient of the
divisor $\Psi(u)$ at $y$.
This means that our $\mathcal{B}_\kappa$ are of this form as well.
Moreover, this representation of $\mathcal{B}_\kappa$ is
compatible with the projection $\RR \times \RR \to \RR$
to the first factor, since in both descriptions 
this projection is induced by inclusion of the torus
acting on $X$ into the torus acting on the special
fibre or by the surjection of the corresponding cocharacter
lattices, respectively.  Then we obtain
\[
F_X(\lambda)
=
\int \lambda u_1 \cdot \deg  \Psi(u_1) du_1
=
\int_{\Delta_y} \lambda u_1 du_1
=
\int_{\mathcal B_\kappa} \lambda u_1 du_1
=
\vol(\mathcal B_\kappa) b_{\kappa,1}.
\]
Hence, $F_X$ vanishes if and only if $b_{\kappa,1}=0$ holds.
We turn to the case of non-product special test configurations.
As, for instance, shown in~\cite[Lemma~3.2.9]{Donaldson},
the, Donaldson-Futaki invariant is given up to positive
constant  by  
\[
\langle \lambda , b_\kappa \rangle
\ = \
\int_{\mathcal  B_\kappa} \langle \lambda , u  \rangle du,
\]
where $\lambda$ is the one-parameter subgroup induced
by the $\CC^*$-action of the test configuration.
By Remark~\ref{rem:antitropical-oneparameter}~(ii)
the second anticanonical coordinate of $\lambda$
is positive.
Thus, $b_{\kappa,1}=0$ implies that we have a positive
Donaldson-Futaki invariant if and only if $b_{\kappa,2}>0$.
\end{proof}


\begin{remark}
In the situation of Theorem~\ref{thm:ke-crit},
we have $b_{0,1} = \ldots = b_{r,1}$ for the first
coordinates of the barycenters.
In particular, it suffices to test the condition
$b_{\kappa,1} = 0$ just for one $\kappa$.
\end{remark}

\begin{example}
\label{example:lbsp6}
Consider the surface $X = X(A,P,\Sigma)$ provided by
Examples~\ref{example:lbsp1} and~\ref{example:lbsp2}.
As noted in~\ref{example:lbsp4} the anticanonical
class of $X$ equals $w_{01}+w_{02}+w_{21}$.
Thus, for the defining matrix of the anticanonical cone
$X'$ over $X \subseteq Z$ we can take 
$$
P'
\ = \
\left[
\begin{array}{rrrrr}
-2 & -1 & 1 & 1 & 0
\\
-2 & -1 & 0 & 0 & 2
\\
3  & -1 & 0 & -1 & 1
\\
1  &  1 & 0 &  0 & 1
\end{array}
\right].
$$
We have $X' \subseteq Z'$, where $Z'$ is the affine
toric variety defined by cone $\sigma' \subseteq \QQ^{s+2}$
generated by the columns of $P'$.
The convergency cones $\tau_\kappa' = \eta_0^{-1}(\sigma')$
of the affine toric varieties
$\mathcal{X}'_{\kappa,0} \subseteq Z'$
are given in antitropical coordinates by
\[
\begin{array}{lcl}
\tau_0' & = & \cone((-1, 1, -1), \, (-1, 1, 2), \, (1, 1, 2), \, (3, 1, -2)),
\\[.5em]
\tau_1'  & = & \cone((-1, 0, -1),  \, (-1, 3, 2),  \, (0, 0, -1),  \, (2, 1, 1)),
\\[.5em]
\tau_2' & = & \cone((-2, 1, 1),  \, (1, 1, -2),  \, (1, 1, 2),  \, (3, 1, 2)),
\end{array}
\]
see Propostion~\ref{prop:fibers}~(iii). The respective dual cones of the
$\tau_\kappa'$ are given as
\[
\begin{array}{lcl}
\omega'_0 & = & \cone((0, 2, -1), \, (1, 1, 0), \, (-2, 4, -1), \,  (1, 5, 4)),
\\[.5em]
\omega'_1 & = & \cone((0, 1, 0), \,  (1, 1, -1), \, \, (-1, 2, 0), \, (1, 5, -7)),
\\[.5em]
\omega'_2 & = & \cone((1, 5, -3), \,  (1, 1, 1), \, (0, 2, -1), \, (-2, 4, 1)).
\end{array}
\]
For the corresponding moment polytopes and their barycenters,
we obtain
\[
\begin{array}{lclcl}
\mathcal{B}_0
& = &
\conv\left(
\left(0, -\frac{1}{2}\right), \,
(1, 0), \,
\left(-\frac{1}{2}, -\frac{1}{4}\right), \,
\left(\frac{1}{5}, \frac{4}{5}\right)
\right),
& & \left(\frac{41}{190}, \frac{79}{1140}\right),
\\[.5em]
\mathcal{B}_1
& = &
\conv\left(
(0, 1), \,
(1, 0), \,
\left(-\frac{1}{2}, 1 \right), \,
\left(\frac{1}{5}, -\frac{2}{5}\right)
\right),
& & \left(\frac{41}{190}, \frac{92}{285}\right),
\\[.5em]
\mathcal{B}_2
& = &
\conv\left(
\left(\frac{1}{5}, -\frac{3}{5}\right), \,
(1, 1), \,
\left(0, -\frac{1}{2}\right), \,
\left(-\frac{1}{2}, \frac{1}{4}\right)
\right),
& & \left(\frac{41}{190}, \frac{217}{1140}\right).                             
\end{array}
\]
Recall from Example~\ref{example:lbsp4}, that only
$\kappa=0,2$ are special.
As we have $b_{0,1} \ne 0$, Theorem~\ref{thm:ke-crit}
shows that the
del Pezzo $\CC^*$-surface $X$ from~\ref{example:lbsp1}
and~\ref{example:lbsp2} doesn't admit a K\"ahler-Einstein
metric.
\end{example}

The treatment of K\"ahler-Ricci solitons
requires a generalized notion of
equivariant $K$-polystability.
Again, let $(X,L)$ be a polarized Fano $T$-variety
as in Remark~\ref{rem:fanopolar}.
Given any element $\xi \in \Lambda_\RR(T)$,
one defines the \emph{modified Futaki character}
associated with $(X,L)$ and $\xi$ to be 
$$
F_{X,\xi}
\ := \
- \lim_{k \to \infty}
\frac{1}{k h_k}
\sum_{\chi \in \Chi(T)}
h_{k,\chi} e^{\frac{\bangle{\chi,\xi}}{k}} \chi
\ \in \
\mathbb{X}_\RR(T)
$$
where $h_k$ and $h_{k,\chi}$ are as before.
Let $\lambda \in \Lambda(\mathcal{T})$ be the
one parameter subgroup corresponding to
the $\CC^*$-part of $\mathcal{T} = T \times \CC^*$.
The \emph{modified Donaldson-Futaki invariant}
of a special equivariant test configuration
$(\mathcal{X},\mathcal{L},\psi)$ of $(X,L)$ is
$$
\mathrm{DF}_\xi (\mathcal{X},\mathcal{L},\psi)
\ := \
F_{\mathcal{X}_0,\xi}(\lambda).
$$
We say that $(X,\xi)$ is \emph{equivariantly $K$-polystable} if
$\mathrm{DF}_\xi (\mathcal{X},\mathcal{L},\psi) \ge 0$
holds for every special equivariant test configuration
with equality precisely in the case of product configurations.

\begin{theorem}
\label{thm:kstab-soliton}
A Fano variety $X$ with torus action admits a
K\"ahler-Ricci soliton if and only if $(X,\xi)$
is equivariantly $K$-polystable
for the $\xi \in \Lambda_\RR(T)$ with $F_{X,\xi} = 0$.
\end{theorem}

\begin{proof}
Use~\cite[Thm.~1.3]{BlLiXuZh} and~\cite{HaLi}.
\end{proof}

\begin{theorem}
\label{thm:krs-crit}
A del Pezzo $\CC^*$-surface $X(A,P,\Sigma)$ admits a
K\"ahler-Ricci soliton if and only if there is a
$\xi \in \RR$, such that for every
special $0 \le \kappa \le r$ we have
\[
\int_{\mathcal{B}_\kappa} u_1 e^{\xi u_1} du_1du_2
\ = \
0,
\qquad\qquad
\int_{\mathcal{B}_\kappa} u_2 e^{\xi u_1} du_1du_2
\ > \
0.
\]
\end{theorem}

\begin{proof}
This follows from Theorem~\ref{thm:kstab-soliton} by using
the integral formula for the modified Futaki character in
the toric case, which can be found e.g. in~\cite{WZZ},
see also \cite[Cor.~4.11]{IlSu}.
\end{proof}

\begin{example}
\label{example:lbsp7}
Consider $X = X(A,P,\Sigma)$ from
Examples~\ref{example:lbsp1} and~\ref{example:lbsp2}.
Example~\ref{example:lbsp6} provides 
the polytopes $\mathcal{B}_0$ and $\mathcal{B}_2$
for the special test configurations.
The first identity of Theorem~\ref{thm:krs-crit}
is satisfied by a unique $\xi \in \RR$ and
we obtain
\[
2.4984  \ \le \ \xi \ \le \  2.4988
\]  
for both cases, $\kappa=0$ and $\kappa=2$,
using interval arithmetic as described in~\cite{CaSu}.
Feeding this estimate into the calculation of the
integral in the scond identity, gives us the bounds
\[
0.0009 \leq \int_{\mathcal{B}_0} u_2 e^{\xi u_1} du_1du_2 \leq 0.0010,
\qquad
0.0797 \leq \int_{\mathcal{B}_2} u_2 e^{\xi u_1} du_1du_2 \leq 0.0799.
\]
This guarantees a positive value in both cases.
Hence, Theorem~\ref{thm:krs-crit} tells us that
the surface  $X = X(A,P,\Sigma)$ from
Examples~\ref{example:lbsp1} and~\ref{example:lbsp2}
admits a K\"ahler-Ricci soliton.
\end{example}

Finally, we come to Sasaki-Einstein metrics on the anticanonical link.
Consider a Fano $T$-variety $X$ with a polarization given by a
linearized $T$-invariant anticanonical divisor $L$ as in
Remark~\ref{rem:fanopolar}. The anticanonical cone over $X$ is the
affine variety
\[
X' \ := \ \Spec \, R(L),
\qquad
R(L)
\ := \
\bigoplus_{k \geq 0} H^0(X, \mathcal{O}(-kL))
\ \cong \
\mathcal{R}(X)([L]).
\]
The natural $T$-linearization of $L$ induces an action of
$T' = T \times \CC^*$ on the anticanonical cone $X'$ and
a grading of its algebra of global functions:
\[
R(L)
\ = \
\bigoplus_{\XX(T) \times \ZZ} R(L)_u .
\]
The weight cone
$\vartheta(X') \subseteq \XX_\RR(T') = \XX_\RR(T) \times \RR$
of this action is pointed and of full dimension.
The \emph{Reeb cone} of $X'$, also of its link,
is given by the dual cone
\[
\tau(X')
\ = \
\vartheta(X')^\vee
\ \subseteq \
\Lambda_\RR(T')
\ = \ 
\Lambda_\RR(T)\times \RR.
\]  
Any (not necessarily rational) $\xi \in \tau(X')^\circ$ is called
a \emph{polarization} of $X'$ and its \emph{log discrepancy}
in the sense of~\cite{JM} is $A_{X'}(\xi):=\langle \xi , (0,1) \rangle$.
Following~\cite[2.5]{LiXu} we set
\[
\vol_{X'}(\xi)
\ := \
\dim(X')!
\lim_{k\to \infty}
\sum_{\genfrac{}{}{0pt}{1}{u \in \XX(T) \times \ZZ}{\langle u, \xi \rangle \leq k}}
\frac{\dim(R_u)}{k^{\dim(X')}}.
\]
This volume function is differentiable in $\xi$.
Moreover, for every
$\lambda \in \Lambda_\RR(T) \times \RR$,
we can define
\[
F_{X',\xi}(\lambda)
\ : = \
\left.\frac{d \vol(\xi - s \cdot (A_{X'}(\xi)\lambda-A_{X'}(\lambda)\xi))}{ds}\right|_{s=0}.
\]

Now, let $(\mathcal{X}', \mathcal{L}', \psi')$
be an equivariant test configuration 
for $(X',L')$, where $L'$ is the naturally
$T'$-linearized (trivial) anticanonical divisor
of $X'$ and let $\lambda \in \Lambda(\mathcal{T}')$ 
correspond to the $\CC^*$-part of
$\mathcal{T}' = T' \times \CC^*$.
Then 
\[
\operatorname{DF}_\xi(\mathcal{X}',\mathcal{L}',\psi')
\ := \
F_{\mathcal{X}'_0,\xi}(\lambda),
\]
is the \emph{Donaldson-Futaki invariant} for 
$(\mathcal{X}', \mathcal{L}', \psi')$.
We call $(X',\xi)$ \emph{equivariantly $K$-polystable} if
$\mathrm{DF}_\xi (\mathcal{X}',\mathcal{L}',\psi') \ge 0$
holds for every special equivariant test configuration
with equality precisely in the case of product configurations.

\begin{theorem}
\label{thm:k-stab-sasaki}
If a Fano variety with torus action admits a Sasaki-Einstein metric
with Reeb vector $\xi$ on its anticanonical cone link,
then its anticanonical cone is equivariantly $K$-polystable
with respect to the polarization given by $\xi$.
\end{theorem}

\begin{proof}
For the case of a smooth link, or equivalently 
an isolated singularity in the cone, this
is~\cite[Thm.~7.1]{ColSze}. In the general case,
\cite[Cor.~A.4]{LWX} gives the assertion.
\end{proof}

\begin{remark}
In the case of a smooth link, \cite[Thm.~1.1]{ColSze}
provides also the converse of the above theorem.
\end{remark}

\begin{remark}
In the setting of Construction~\ref{constr:moment-polytope},
the Reeb cone of $X'$ equals the defining cone 
of the ambient toric variety $Z'$ of $X'$
and hence equals the cone $\sigma' \subseteq \RR^{s+2}$
generated by the columns of $P'$.
\end{remark}

\begin{construction}
\label{momentpol}
Consider a Fano variety $X = X(A,P,\Sigma)$
with $\Sigma$ polytopal.
Given a special $0 \le \kappa \le r$,
we have the mutually dual cones
\[
\tau_\kappa \ \subseteq \ \RR^{s'+1},
\qquad
\omega_\kappa \ \subseteq \ \RR^{s'+1}
\]
introduced in Proposition~\ref{prop:conemeetsdeg}~(v).
With any vector 
$\xi \in (\tau_\kappa \cap \RR^s\times \{1\} \times \RR)^\circ$
we associate the polytope
\[
\mathcal{B}'_\kappa(\xi)
\ := \
\omega_\kappa \cap [\langle \xi, \cdot \rangle \leq 1]
\ \subseteq \
\RR^{s'+1}.
\]
\end{construction}

\begin{theorem}
\label{thm:se-crit}
If a del Pezzo $\CC^*$-surface $X = X(A,P,\Sigma)$
admits a Sasaki-Einstein metric on its
anticanonical cone link then there is a vector
$\xi=(\xi_1,1,0)$ in
$(\tau_\kappa \cap \RR^s\times \{1\}\times\{0\})^\circ$
such that for every special $0 \leq \kappa \leq r$
we have
\[
\left.\frac{d \vol(\mathcal{B}'_\kappa(\xi_1-s,1,0))}{ds}\right|_{s=0}
\ = \
0,
\qquad
\left.\frac{d \vol(\mathcal{B}'_\kappa(\xi_1,1,-s))}{ds}\right|_{s=0}
\ > \
0.
\]
\end{theorem}

\begin{proof}
This follows from Theorem~\ref{thm:k-stab-sasaki}
by using the identity
$\vol_{\mathcal{X}'_0}(\xi)=\vol(\mathcal{B}'_\kappa(\xi))$,
which appears in~\cite[3.2.1]{LiXu}, see
also~\cite[Thm.~4.10.]{Suess}.
\end{proof}

In our computations, we worked with the following slightly
more explicit formulation of Theorem~\ref{thm:se-crit}.

\begin{remark}
\label{rem:se-crit}
Given a del Pezzo $\mathbb{K}^*$-surface $X=X(A,P)$ and
$0 \le \kappa \le r$,
consider the cone $\tau_\kappa$ and the polytope
$\mathcal{B}_\kappa'$ from Construction~\ref{momentpol}.
Then we have a rational function on
$\tau_\kappa \cap \RR \times \{1\} \times \RR$
given by 
\[
\vol_\kappa(x_1,x_2)
\ := \
\vol(\mathcal{B}_\kappa'(x_1,1,x_2)).
\]
Let $z$ be the unique zero with $(z,1,0) \in \tau'_\kappa$
of 
$x_1 \mapsto \partial\vol_\kappa / \partial x_1 \, (x_1,1,0)$.
If the anticanonical cone link of $X$ arising
from Proposition~\ref{prop:affinecone}
admits a Sasaki-Einstein metric, then 
$\partial\vol_\kappa / \partial x_2 \, (z,1,0) < 0$
for every special $\kappa = 0, \ldots, r$.
\end{remark}

\begin{example}
\label{example:lbsp8}
Consider once more $X = X(A,P,\Sigma)$ from~\ref{example:lbsp1}
and~\ref{example:lbsp2}.
As in Example~\ref{example:lbsp6}, we take anticanonical cone $X'$
over $X$ arising from
$$
P'
\ = \
\left[
\begin{array}{rrrrr}
-2 & -1 & 1 & 1 & 0
\\
-2 & -1 & 0 & 0 & 2
\\
3  & -1 & 0 & -1 & 1
\\
1  &  1 & 0 &  0 & 1
\end{array}
\right].
$$
We look at $\kappa=0$, which is special.
The volume function defined by $\mathcal{B}_0'(\xi)$
for $\xi=(\xi_1,1,\xi_3)$ is explicitly given as 
\[
\vol (\mathcal{B}_0'(\xi))
\ = \
\frac{4}{{\left(2 \xi_{1} + \xi_{3} - 4\right)} {\left(\xi_{1} + 1\right)} {\left(\xi_{3} - 2\right)}}
-
\frac{28}{{\left(2 \xi_{1} + \xi_{3} - 4\right)} {\left(\xi_{1} + 4 \xi_{3} + 5\right)} {\left(\xi_{1} + 1\right)}}.
\]
The unique $(\xi_1,1,0) \in \tau_0$ minimizing this volume function
is given by the unique zero $\xi_1^0$ of the partial derivative
w.r.t. $\xi_1$ satisfying $(\xi_1^0,1,0) \in \tau_0$. We obtain
\[
0.64082  < \xi_1^0 < 0.64096,
\qquad
0.00923
< 
\frac{d \vol(\mathcal{B}'_0(\xi_1,1,\xi_3))}{d\xi_3} (\xi_1^0,1,0)
<
0.00963,
\]
using numerical methods and interval arithmetic.
Thus, Theorem~\ref{thm:se-crit} tells us that
there is no Sasaki-Einstein metric on the
anticanonical cone link of $X = X(A,P,\Sigma)$.
%
%
\end{example}

\begin{proof}[Proof of Theorems~\ref{thm:main1} and~\ref{thm:main2}].
The defining data $(A,P)$ for the log del Pezzo $\CC^*$-surfaces in
question are provided by the classifications~\cite{HH,HHHS};
see also the database~\cite{HHS}.
For any given defining matrix $P$,
the criteria for existence of K\"ahler-Einstein
metrics~\ref{thm:ke-crit},
K\"ahler-Ricci solitons~\ref{thm:krs-crit}
and Sasaki-Einstein metrics~\ref{thm:se-crit}
can be checked computationally, following the
lines of the Examples~\ref{example:lbsp6},
\ref{example:lbsp7} and~\ref{example:lbsp8}.
\end{proof}



\end{document}